\documentclass[10pt]{article}

\usepackage{color} 
\usepackage{epsfig}

\usepackage{amsmath,amssymb,amsfonts,amsthm}
\usepackage{moreverb,rotating,graphics}
 \usepackage[T1]{fontenc}
 \usepackage[normalem]{ulem}
 \usepackage{verbatim}
\usepackage[utf8]{inputenc}

\usepackage{latexsym}
\usepackage{pdflscape}
\usepackage{examplep}
\usepackage{longtable}
\usepackage{graphicx}
 \usepackage{psfrag}
 \usepackage{epsfig}
\usepackage{amsmath}
\usepackage{booktabs}
\usepackage{memhfixc}

\usepackage[francais,english]{babel} 

\newtheorem{theorem}{Theorem}[section]
\newtheorem{proposition}{Proposition}[section]
\newtheorem{remark}{Remark}[section]
\newtheorem{lemma}{Lemma}[section]

\newtheorem{definition}{Definition}[section]
\newtheorem{example}{Example}[section]

\def\cC{{\mathcal C}}

\def\Vt{{\widetilde{V}}}
\def\P{\mathbb{P}}

\def\N{{\mathbb N}}
\def\Z{{\mathbb Z}}
\def\R{{\mathbb R}}

\def\P{{\mathbb P}}

\def\cH{{\cal H}}

\def\cC{{\cal C}}
\def\cL{{\cal L}}
\def\cM{{\cal M}}

\def\cS{{\cal S}}

\def\cV{{\cal V}}
\def\cX{{\cal X}}
\def\cE{{\cal E}}

\def\cJ{{\cal J}}

\def\Vt{\widetilde{V}}

\def\Sigmat{\widetilde{\Sigma}}
\def\zt{\widetilde{z}}
\def\ut{\widetilde{u}}

\def\ut{\widetilde{u}}

\def\rt{\widetilde{r}}
\def\st{\widetilde{s}}
\def\Vt{\widetilde{V}}
\def\cD{\mathcal{D}}

\def\cA{\mathcal{A}}
\def\cB{\mathcal{B}}
\def\P{\mathbb{P}}

\newcommand \dps{\displaystyle }

\title{Greedy algorithms for high-dimensional eigenvalue problems}
\author{Eric Canc\`es, Virginie Ehrlacher, Tony Leli\`evre}

\begin{document}

\selectlanguage{english}

\maketitle

\begin{abstract} In this article, we present two new greedy algorithms for the computation of the lowest eigenvalue (and an associated eigenvector)
 of a high-dimensional eigenvalue problem, and prove some convergence results for these algorithms and their orthogonalized versions. 
The performance of our algorithms is illustrated on numerical test cases (including the computation of the buckling modes of a microstructured plate),
 and compared with that of another greedy algorithm for eigenvalue problems introduced by Ammar and Chinesta.    
\end{abstract}


\section{Introduction}

High dimensional problems are encountered in many application fields, among which electronic structure calculation, molecular dynamics, uncertainty quantification, multiscale homogenization, and mathematical finance. The numerical simulation of these problems, which requires specific approaches due to the so-called \itshape curse of dimensionality\normalfont~\cite{Bellman}, has fostered the development of a wide variety of new numerical methods and algorithms, such as sparse grids~\cite{Griebel, SchwabPeter}, reduced bases~\cite{Maday}, sparse tensor products~\cite{Hackbusch}, and adaptive polynomial approximations~\cite{Cohen}. 

In this article, we focus on an approach introduced by Ladev\`eze~\cite{Ladeveze}, Chinesta~\cite{Chinesta}, Nouy~\cite{Nouy} and coauthors in different contexts, 
relying on the use of \itshape greedy algorithms\normalfont~\cite{Temlyakov}. This class of methods is also called \itshape Progressive Generalized Decomposition\normalfont
~\cite{Ladevezeref} in the literature. 
 
Let $V$ be a Hilbert space of functions depending on $d$ variables $x_1 \in \cX_1, \; \ldots ,\;  x_d \in \cX_d$, where, typically, $\cX_j \subset \R^{m_j}$. For all $1\leq j \leq d$, let $V_j$ be a Hilbert space of functions depending 
only on the variable $x_j$ such that for all $d$-tuple $\left( \phi^{(1)}, \ldots , \phi^{(d)} \right)\in V_1\times \cdots \times V_d$, the tensor-product function 
$\phi^{(1)} \otimes \cdots \otimes \phi^{(d)}$ defined by
$$
\phi^{(1)} \otimes \cdots \otimes \phi^{(d)} \; : \; \left\{ \;
\begin{array}{ccc} 
 \cX_1 \times \cdots \times \cX_d & \to & \R \\
(x_1, \ldots , x_d) & \mapsto & \phi^{(1)}(x_1) \cdots \phi^{(d)}(x_d)\\
\end{array}
\right .
$$
belongs to $V$. Let $u$ be a specific function of $V$, for instance the solution of a Partial Differential Equation (PDE). Standard linear approximation approaches such as Galerkin methods consist in approximating the 
function $u(x_1, \ldots, x_d)$ as
\begin{align*}
u(x_1, \ldots , x_d) & \approx  \sum_{1\leq i_1, \ldots, i_d \leq N} \lambda_{i_1, \ldots , i_d} \phi_{i_1}^{(1)}(x_1) \cdots \phi_{i_d}^{(d)}(x_d)\\
&  =  \sum_{1\leq i_1, \ldots , i_d \leq N} \lambda_{i_1, \ldots , i_d} \phi_{i_1}^{(1)}\otimes \cdots \otimes \phi_{i_d}^{(d)}(x_1, \ldots , x_d),\\
\end{align*}
where $N$ is the number of degrees of freedom per variate (chosen the same for each variate to simplify the notation), and where for all $1\leq j \leq d$, 
$\left( \phi_i^{(j)} \right)_{1\leq i \leq N}$ is an \itshape a priori chosen  \normalfont discretization basis of functions belonging to $V_j$. 
To approximate the function $u$, the set of $N^d$ real numbers $\left( \lambda_{i_1, \ldots , i_d} \right)_{1\leq i_1, \ldots , i_d \leq N}$ 
must be computed. Thus, the size of the discretized problem to solve scales exponentially with $d$, the number of variables. Because of this difficulty,
 classical methods cannot be used in practice 
to solve high-dimensional PDEs. Greedy algorithms also consist in approximating the function $u(x_1, \ldots, x_d)$ as a sum of tensor-product functions
$$
u(x_1, \ldots, x_d) \approx u_n(x_1, \ldots, x_d)=\sum_{k=1}^n r_k^{(1)}(x_1)\cdots r_k^{(d)}(x_d) = \sum_{k=1}^n r_k^{(1)}\otimes \cdots \otimes r_k^{(d)}(x_1,\ldots, x_d),
$$
where for all $1\leq k \leq n$ and all $1\leq j \leq d$, $r_k^{(j)} \in V_j$. But in contrast with standard linear approximation methods, the sequence of tensor-product functions $\left( r_k^{(1)}\otimes \cdots \otimes r_k^{(d)} \right)_{1\leq k \leq n}$ 
is not chosen {\it a priori}; it is constructed iteratively using a greedy procedure. Let us illustrate this on the simple case when the function $u$ to be computed is the unique solution of a minimization problem of the form
\begin{equation}\label{eq:minpb}
u = \mathop{\mbox{\rm argmin}}_{v\in V} \cE(v),
\end{equation}
where $\cE: V\to \R$ is a strongly convex functional. Denoting by 
$$
\Sigma^\otimes := \left\{ r^{(1)} \otimes \cdots \otimes r^{(d)}\; | \; r^{(1)} \in V_1, \; \ldots, \; r^{(d)} \in V_d\right\}
$$ 
the set of rank-1 tensor-product functions, the \itshape Pure Greedy Algorithm \normalfont (PGA) for solving (\ref{eq:minpb})  reads
\begin{itemize}
 \item \bfseries Initialization: \normalfont set $u_0  := 0$;  
\item \bfseries Iterate on $n\geq 1$: \normalfont find $z_n:=  r_n^{(1)} \otimes \cdots \otimes r_n^{(d)} \in \Sigma^\otimes$ such that
$$
z_n \in \mathop{\mbox{\rm argmin}}_{z \in \Sigma^\otimes} \cE\left( u_{n-1} + z \right),
$$
and set $u_n := u_{n-1} + z_n$.
\end{itemize}
The advantage of such an approach is that if, as above, a discretization basis $\left( \phi^{(j)}_i \right)_{1\leq i \leq N}$ is used for the approximation of 
the function $r_n^{(j)}$, each iteration of the algorithm requires the resolution of a discretized problem of size $dN$.
 The size of the problem to solve at iteration $n$ therefore scales linearly with the number of variables. Thus, 
using the above PGA enables one to approximate the function $u(x_1, \ldots , x_d)$ through the resolution of a sequence of low-dimensional problems, instead of one high-dimensional problem.  

\medskip

Greedy algorithms have been extensively studied in the framework of problem (\ref{eq:minpb}). The PGA has been analyzed from a mathematical point of view, 
firstly in~\cite{LBLM} in the case when $\cE(v) := \|v-u\|_V^2$, then in~\cite{CELgreedy} in the case of a more general nonquadratic strongly convex energy functional $\cE$. In the latter article, it is proved that the sequence $(u_n)_{n\in\N^*}$ strongly converges in $V$ to $u$, 
provided that: (i) $\Sigma^\otimes$ is weakly closed in $V$ and  $\mbox{\rm Span}\left(\Sigma^\otimes\right)$ is dense in $V$; 
(ii) the functional $\cE$ is strongly convex, differentiable on $V$, and its derivative is Lipschitz on bounded domains. 
An exponential convergence rate is also proved in the case when $V$ is finite-dimensional. In~\cite{NouyFalco}, these results have been extended to the case when general tensor subsets $\Sigma$ are considered instead of the set of rank-1 
tensor-products $\Sigma^{\otimes}$, and under weaker assumptions on the functional $\cE$. The authors also generalized the convergence results to other variants 
of greedy algorithms, like the \itshape Orthogonal Greedy Algorithm \normalfont (OGA), and to the case when the space $V$ is a Banach space. 

\medskip

The analysis of greedy algorithms for other kinds of problems is less advanced~\cite{Ladevezeref}. We refer to~\cite{nonsym} for a review of 
the mathematical issues arising in the application of greedy algorithms to non-symmetric linear problems. To our knowledge, 
the literature on greedy algorithms for eigenvalue problems is very limited. Penalized formulations of constrained minimization problems enable one to recover the structure of unconstrained 
minimization problems and to use the existing theoretical framework for the PGA and the OGA~\cite{CELgreedy,NouyFalcopen}. The only reference we are aware of 
about greedy algorithms for eigenvalue problems without the use of a penalized formulation is an article by Ammar and Chinesta~\cite{Chinesta-QC}, 
in which the authors propose a greedy algorithm to compute the lowest eigenstate of a bounded from below self-adjoint operator, and apply it to electronic structure calculation. 
No analysis for this algorithm is given though. Let us also mention that the use of tensor formats for eigenvalue problems has been recently 
investigated~\cite{Hackbusch, Schneider, Beylkin1, Beylkin2, Khoromskij}, still in the context of electronic structure calculation.

\medskip

In this article, we propose two new greedy algorithms for the computation of the lowest eigenstate of high-dimensional eigenvalue problems and prove some 
convergence results for these algorithms and their orthogonalized versions. We would like to point out that these algorithms are not based on a penalized 
formulation of the eigenvalue problem.

The outline of the article is as follows. In Section~\ref{sec:prel}, we introduce some notation, give some prototypical examples 
of problems and tensor subsets for which our analysis is valid, and recall earlier results on greedy algorithms for unconstrained convex minimization problems.
 In Section~\ref{sec:mainsec}, the two new approaches are presented along with our main convergence results.
 The first algorithm is based on the minimization of the Rayleigh quotient associated to 
the problem under consideration. The second method relies on the minimization of 
a residual associated to the eigenvalue problem. Orthogonalized versions of these algorithms are also introduced. In Section~\ref{sec:resnum}, we detail how these algorithms can be implemented in practice in the case of rank-1 tensor product functions. 
The numerical behaviors of our algorithms and of the one proposed in~\cite{Chinesta-QC} are illustrated in Section~\ref{sec:numtest}, 
first on a toy example, then on the computation of the buckling modes of a microstructured plate. 
The proofs of our results are given in Section~\ref{sec:proof}. Finally, some pathological cases are discussed in the Appendix. Let us mention that we do not cover here 
the case of parametric eigenvalue problems which will make the object of a forthcoming article.

\section{Preliminaries}\label{sec:prel}

\subsection{Notation and main assumptions}\label{sec:not}

Let us consider two Hilbert spaces $V$ and $H$, endowed respectively with the scalar products $\langle \cdot, \cdot \rangle_V$ and 
$\langle \cdot, \cdot \rangle$, such that, unless it is otherwise stated, 
\begin{itemize}
 \item [(HV)] the embedding $V \hookrightarrow H$ is dense and compact. 
\end{itemize}
The associated norms are denoted respectively by $\|\cdot\|_V$ and $\|\cdot\|$. Let us recall that it follows from (HV) that the weak convergence in $V$ implies the strong convergence in $H$. 

Let $a:V\times V \to \R$ be a symmetric continuous bilinear form on $V\times V$ such that
\begin{itemize}
 \item [(HA)] $\exists \gamma, \nu >0, \; \mbox{ such that }\; \forall v\in V, \; a(v,v) \geq \gamma \|v\|_V^2 - \nu \|v\|^2. $
\end{itemize}

The bilinear form $\langle \cdot, \cdot \rangle_a$, defined by
\begin{equation}\label{eq:ascal}
\forall v,w\in V, \quad \langle v, w\rangle_a := a(v,w) + \nu\langle v, w\rangle,
\end{equation}
is a scalar product on $V$, whose associated norm, denoted by $\|\cdot\|_a$, is equivalent to the norm $\|\cdot\|_V$. Besides, we can also assume without loss of generality that the constant $\nu$ is chosen 
so that for all $v\in V$, $\|v\|_a \geq \|v\|$.

\medskip

It is well-known (see e.g. \cite{ReedSimon4}) that, under the above assumptions (namely (HA) and (HV)), 
there exists a sequence $(\psi_p, \mu_p)_{p\in\N^*}$ of solutions to the elliptic eigenvalue problem
\begin{equation}\label{eq:eigenvalue}
\left\{
\begin{array}{l}
 \mbox{find }(\psi,\mu)\in V\times \R \mbox{ such that }\|\psi\|=1 \mbox{ and }\\
\forall v\in V, \; a(\psi,v) = \mu\langle \psi,v \rangle \\
\end{array}
\right .
\end{equation}
such that $(\mu_p)_{p\in\N^*}$ forms a non-decreasing sequence of 
real numbers going to infinity and $(\psi_p)_{p\in\N^*}$ is an orthonormal basis of $H$. 
We focus here on the computation of $\mu_1$, the lowest eigenvalue of $a(\cdot, \cdot)$, and of an associated $H$-normalized eigenvector. Let us note that, from (HA), for all 
$p\in \N^*$, $\mu_p + \nu >0$. 

\medskip

In the case when the embedding $V\hookrightarrow H$ is dense but not compact (i.e. when (HV) does not hold), the spectrum of the unique self-adjoint operator $A$ on $H$ with form domain $V$ associated 
with the quadratic form $a(\cdot, \cdot)$ can be purely continuous; in this case, (\ref{eq:eigenvalue}) has no solution. However, if $A$ has at least one discrete eigenvalue 
located below the minimum of its essential spectrum, convergence results for the second algorithm we propose can be established. This is the object of Proposition~\ref{prop:noHV}.

\begin{definition}\label{def:dictionary}
 A set $\Sigma \subset V$ is called a \normalfont dictionary \itshape of $V$ if $\Sigma$ satisfies the following three conditions: 
\begin{description}
 \item[(H$\Sigma 1$)] $\Sigma$ is a non-empty cone, i.e. $0\in \Sigma$ and for all $(z,t)\in \Sigma\times \R$, $tz\in \Sigma$; 
\item[(H$\Sigma 2$)] $\Sigma$ is weakly closed in $V$; 
\item[(H$\Sigma 3$)] $\mbox{\rm Span}(\Sigma)$ is dense in $V$.  
\end{description}
\end{definition}

In practical applications for high-dimensional eigenvalue problems, 
the set $\Sigma$ is typically an appropriate set of tensor formats used to perform the greedy algorithms presented in Section~\ref{sec:algo}. We also denote by 
\begin{equation}\label{eq:defsigstar}
\Sigma^*:= \Sigma \setminus\{ 0\}. 
\end{equation}

\subsection{Prototypical example}\label{sec:example}

Let us present a prototypical example of the high-dimensional eigenvalue problems we have in mind, along with possible dictionaries. 

Let $\cX_1,\; \ldots, \;\cX_d$ be bounded regular domains of $\R^{m_1}, \;\ldots ,\; \R^{m_d}$ respectively. Let $V = H^1_0(\cX_1 \times \cdots \times \cX_d)$ 
and $H= L^2(\cX_1\times \cdots \times \cX_d)$. It follows from the Rellich-Kondrachov theorem that these spaces satisfy assumption (HV). 
Let $b:\cX_1 \times \cdots \times \cX_d \to \R$ be a measurable real-valued function such that 
$$
\exists \beta, B>0, \; \mbox{ such that } \; \beta \leq b(x_1, \ldots, x_d) \leq B, \; \mbox{ for a.a. }(x_1, \ldots, x_d)\in \cX_1\times \cdots \times \cX_d.
$$
Besides,  let $W\in L^q(\cX_1\times \cdots \times \cX_d)$ with $q = 2$ if $m\leq 3$, and $q > m/2$ for $m\geq 4$ where $m := m_1 + \cdots + m_d$.
A prototypical example 
of a continuous symmetric bilinear form $a: V\times V \to \R$ satisfying (HA) is 
\begin{equation}\label{eq:defaex}
\forall v,w\in V, \quad a(v,w) := \int_{\cX_1\times \cdots \times \cX_d} \left( b \nabla v \cdot \nabla w + W vw\right).
\end{equation}
In this particular case, the eigenvalue problem (\ref{eq:eigenvalue}) also reads
$$
\left\{
\begin{array}{l}
 \mbox{find } (\psi,\mu)\in H^1_0(\cX_1\times \cdots \times \cX_d)\times \R \mbox{ such that }\|\psi\|_{L^2(\cX_1 \times \cdots \times \cX_d)} = 1 \mbox{ and }\\
-\mbox{div}\left( b \nabla \psi \right) + W\psi = \mu \psi \mbox{ in }\cD'(\cX_1\times \cdots \times \cX_d).\\
\end{array}
 \right .
$$ 
 
For all $1\leq j \leq d$, we denote by $V_j:= H^1_0(\cX_j)$. 
Some examples of dictionaries $\Sigma$ based on different tensor formats satisfying (H$\Sigma 1$), (H$\Sigma 2$) and (H$\Sigma 3$) are
 the set of rank-$1$ tensor-product functions
\begin{equation}\label{eq:rank1}
\Sigma^\otimes:= \left\{ r^{(1)} \otimes \cdots \otimes r^{(d)}\;| \; \forall 1\leq j \leq d, \; r^{(j)} \in V_j\right\},
\end{equation}
as well as other tensor formats~\cite{Hackbusch, Khoromskij}, for instance the sets of rank-$R$ Tucker, rank-$R$ Tensor Train, or rank-$R$ Tensor Chain functions, 
with $R\in \N^*$.

\subsection{Greedy algorithms for unconstrained convex minimization problems}\label{sec:convex}

We recall here some results proved in \cite{CELgreedy,Figueroa, LBLM,NouyFalco} on greedy algorithms for convex minimization problems. 
These algorithms are important for our purpose, as they are used to solve subproblems in the strategies we propose
 for the resolution of the eigenvalue problem (\ref{eq:eigenvalue}). 

Let $\cE: V \to \R$ be a real-valued functional defined on $V$ such that
\begin{itemize}
 \item [(HE1)] $\cE$ is differentiable on $V$ and its gradient is Lipschitz on bounded sets, i.e. for all $K\subset V$ bounded subset of $V$, there exists $L_K\in \R_+$ such that
$$
\forall v,w\in K, \quad \|\nabla\cE(v) - \nabla \cE(w)\|_V \leq L_K \|v-w\|_V;
$$
\item[(HE2)] $\cE$ is elliptic, i.e. there exist $\eta>0$ and $s>1$ such that 
$$
\forall v,w\in V, \quad \left\langle \nabla \cE(v) - \nabla \cE(w), v-w \right\rangle_V \geq \eta\|v-w\|^s.
$$ 
\end{itemize}

Then, the functional $\cE$ is strictly convex and the minimization problem
\begin{equation}\label{eq:minc}
\mbox{ \rm find }u\in V \mbox{ such that }u \in \mathop{\rm argmin}_{v\in V} \cE(v),
\end{equation}
has a unique solution. The \itshape Pure Greedy Algorithm \normalfont (PGA) and the \itshape Orthogonal Greedy Algorithm \normalfont (OGA)~\cite{Temlyakov} are defined as follows. 
\medskip

\bfseries Pure Greedy Algorithm (PGA): \normalfont

\begin{itemize}
\item \bfseries Initialization: \normalfont set $u_0 = 0$;
\item \bfseries Iterate on $n\geq 1$: \normalfont find $z_n\in \Sigma$ such that 
\begin{equation}\label{eq:minconv}
z_n \in \mathop{\mbox{argmin}}_{z\in \Sigma} \cE(u_{n-1} + z),
\end{equation}
and set $u_n  := u_{n-1} + z_n$.
\end{itemize}

\medskip

\bfseries Orthogonal Greedy Algorithm (OGA): \normalfont

\begin{itemize}
\item \bfseries Initialization: \normalfont set $u_0 = 0$;
\item \bfseries Iterate on $n\geq 1$: \normalfont find $z_n\in \Sigma$ such that 
\begin{equation}\label{eq:minconvorth}
z_n \in \mathop{\mbox{argmin}}_{z\in \Sigma} \cE(u_{n-1} + z);
\end{equation}
find $\left(c_1^{(n)}, \ldots, c_n^{(n)} \right) \in \R^{n}$ such that
\begin{equation}\label{eq:minorth}
\left( c_1^{(n)}, \ldots, c_n^{(n)}\right) \in \mathop{\rm argmin}_{(c_1, \ldots, c_n) \in \R^n} \cE\left(c_1 z_1 + c_2 z_2 + \cdots + c_nz_n\right),
\end{equation}
and set $u_n  := \sum_{k=1}^n c_k^{(n)}z_k$.
\end{itemize}

\medskip

The following lemma is proved in~\cite{NouyFalco}.
\begin{lemma}\label{lem:lemconvex}
Let $V$ be a separable Hilbert space, $\Sigma$ a dictionary of $V$, and $\cE:V\to \R$ satisfying (HE1) and (HE2). For all $w\in V$, there exists at least one solution to the 
minimization problem: 
$$
\begin{array}{l}
\mbox{find }z_0\in \Sigma \mbox{ such that}\\
z_0 \in \mathop{\mbox{argmin}}_{z\in \Sigma} \cE(w+z).\\
\end{array}
$$
\end{lemma}
 
This lemma implies in particular that all the iterations of the PGA and OGA are well-defined. Besides, the following theorem holds~\cite{NouyFalco}.
\begin{theorem}\label{th:convex}
 Let $V$ be a separable Hilbert space, $\Sigma$ a dictionary of $V$, and $\cE:V\to \R$ satisfying (HE1) and (HE2). 
Then, each iteration of the PGA and OGA is well-defined in the sense that there always exists a solution to the minimization problems (\ref{eq:minconv}), 
(\ref{eq:minconvorth}) and (\ref{eq:minorth}). Besides, the sequence $(u_n)_{n\in\N^*}$ strongly converges in $V$ to $u$, the unique solution of (\ref{eq:minc}).
\end{theorem}

In the case when for all $v\in V$, $\cE(v):= \frac{1}{2}\|v\|_a^2 - \langle l, v\rangle_{V', V}$ for some $l\in V':= \cL(V, \R)$, 
we have the following lemma, proved in~\cite{Temlyakov}.
\begin{lemma}\label{lem:Figueroa}
Let $V$ be a separable Hilbert space, $\Sigma$ a dictionary of $V$, and $\cE:V\to \R$ defined by
$$
\forall v\in V, \quad \cE(v) := \frac{1}{2}\|v\|_a^2 - \langle l, v \rangle_{V', V}
$$
for some $l\in V'$. Then, for all $n\in\N^*$, a vector $z_n\in \Sigma$ solution of (\ref{eq:minconv}) or (\ref{eq:minconvorth}) satisfies
$$
\|z_n\|_a = \mathop{\sup}_{z \in \Sigma^*} \frac{\langle l, z \rangle_{V', V} - \langle u_{n-1}, z\rangle_a}{\|z\|_a}.
$$
In particular, for $n=1$, 
$$
\|z_1\|_a = \mathop{\sup}_{z \in \Sigma^*} \frac{\langle l, z \rangle_{V', V}}{\|z\|_a}.
$$
\end{lemma}

\section{Greedy algorithms for eigenvalue problems} \label{sec:mainsec}

In the rest of the article, we define and study two different greedy algorithms to compute an eigenpair associated to the lowest eigenvalue of the 
elliptic eigenvalue problem (\ref{eq:eigenvalue}).
 
The first one relies on the minimization of the Rayleigh quotient of $a(\cdot, \cdot)$ and is introduced in Section~\ref{sec:rayleigh}. The second one, presented in Section~\ref{sec:residual},
 is based on the use of a residual for problem (\ref{eq:eigenvalue}). We recall the algorithm introduced in~\cite{Chinesta-QC} in Section~\ref{sec:PEGA}. 
Orthogonal versions of these algorithms are defined in Section~\ref{sec:orth}. Section~\ref{sec:main} contains our main convergence results. The choice of a good initial guess for all these algorithms 
is discussed in Section~\ref{sec:initguess}. The proofs of the results stated in this section are postponed until Section~\ref{sec:proof}. 

\subsection{Two useful lemmas}\label{sec:preliminary}

For all $v\in V$, we denote by
$$
\cJ(v):= \left\{
\begin{array}{l}
 \frac{a(v,v)}{\|v\|^2} \; \mbox{ if }v\neq 0,\\
+\infty \; \mbox{ if }v=0,\\
\end{array}
\right .
$$
the Rayleigh quotient associated to (\ref{eq:eigenvalue}), and 
$$
\lambda_\Sigma  := \mathop{\inf}_{z\in \Sigma} \cJ(z) = \mathop{\inf}_{z \in \Sigma^*} \frac{a(z, z)}{\|z\|^2}. 
$$

Note that, since $\Sigma \subset V$, $\dps \lambda_\Sigma \geq \mu_1= \mathop{\inf}_{v\in V} \cJ(v)$.

\begin{lemma}\label{lem:leminit}
Let $w\in V$ such that $\|w\|=1$. The following two assertions are equivalent:
\begin{itemize}
 \item [(i)] $\forall z\in \Sigma, \quad \cJ(w+z) \geq \cJ(w)$;
\item [(ii)] $w$ is an eigenvector of the bilinear form $a(\cdot, \cdot)$ associated to an eigenvalue lower or equal than $\lambda_\Sigma$, i.e. there exists 
$\lambda_w \in \R$, such that $\lambda_w \leq \lambda_\Sigma$ and
$$
\forall v\in V, \; a(w,v) = \lambda_w \langle w, v \rangle.
$$
\end{itemize}
\end{lemma}

\begin{lemma}\label{lem:lemma2}
 Let $w\in V \setminus \Sigma^*$. Then, the minimization problem
\begin{equation}\label{eq:minipb}
\mbox{ find }z_0\in \Sigma \mbox{ such that } z_0 \in \mathop{\mbox{\rm argmin}}_{z\in \Sigma} \cJ(w + z)
\end{equation}
has at least one solution. 
\end{lemma}

When $w\in \Sigma^*$, problem (\ref{eq:minipb}) may have no solution (see Example~\ref{ex:ex1}).

\subsection{Description of the algorithms}\label{sec:algo}

\subsubsection{Pure Rayleigh Greedy Algorithm}\label{sec:rayleigh}

The following algorithm, called hereafter the \itshape Pure Rayleigh Greedy Algorithm \normalfont (PRaGA) algorithm, is inspired from the PGA for convex minimization problems
 (see Section~\ref{sec:convex}). 

\medskip

\bfseries Pure Rayleigh Greedy Algorithm (PRaGA): \normalfont

\begin{itemize}
\item \bfseries Initialization: \normalfont choose an initial guess $u_0\in V$ such that $\|u_0\| = 1$ and such that $\lambda_0:= a(u_0,u_0) < \lambda_\Sigma$;
\item \bfseries Iterate on $n\geq 1$: \normalfont find $z_n\in \Sigma$ such that 
\begin{equation}\label{eq:minRay1}
z_n \in \mathop{\mbox{argmin}}_{z\in \Sigma} \cJ(u_{n-1} + z),
\end{equation}
and set $u_n  := \frac{u_{n-1} + z_n}{\|u_{n-1} + z_n\|}$ and $\lambda_n:=a(u_n, u_n)$.
\end{itemize}

\medskip

Let us point out that in our context, the functional $\cJ$ is not convex, so that the analysis presented for the PGA in Section~\ref{sec:convex} does not hold for the PRaGA.

\medskip

The choice of an initial guess $u_0\in V$ satisfying $\|u_0\| = 1$ and $a(u_0, u_0) \leq \lambda_\Sigma$ is discussed in Section~\ref{sec:initguess}. Let us already mention that for 
the PRaGA (unlike the two other algorithms, the PReGA and the PEGA, presented in the following sections), we require the additional condition that $a(u_0, u_0) < \lambda_\Sigma$ (the inequality is strict). 
We also discuss this point in Section~\ref{sec:initRay}.

\medskip

\begin{lemma}\label{lem:Radef}
Let $V$ and $H$ be separable Hilbert spaces satisfying (HV), $\Sigma$ a dictionary of $V$ and $a: V\times V\to \R$ a symmetric 
continuous bilinear form satisfying (HA). Then, all the iterations of the PRaGA algorithm are well-defined in the sense that for all $n\in\N^*$, there exists at least one solution to the minimization 
problem (\ref{eq:minRay1}). Besides, the sequence $(\lambda_n)_{n\in\N^*}$ is non-increasing. 
\end{lemma}

\begin{proof}
Lemma~\ref{lem:Radef} can be proved reasoning by induction. For $n=1$, since $\|u_0\| =1 $ and $a(u_0, u_0) < \lambda_\Sigma = 
\mathop{\inf}_{z\in \Sigma^*} \frac{a(z,z)}{\|z\|^2}$, necessarily $u_0 \notin \Sigma$. Thus, from Lemma~\ref{lem:lemma2}, (\ref{eq:minRay1}) has at least one solution $z_1\in \Sigma$. 
Besides, if $u_1:= \frac{u_0 + z_1}{\|u_0 + z_1\|}$, we have $\|u_1\| = 1$ and $\lambda_1 = a(u_1,u_1) \leq \lambda_0 = a(u_0,u_0) < \lambda_{\Sigma}$. 
Resaoning by induction, for all $n\in \N^*$, it is clear that $\|u_{n-1}\| = 1$ and $\lambda_{n-1} = a(u_{n-1}, u_{n-1}) < \lambda_0$. Thus, $u_{n-1} \notin \Sigma$ and using the same kind of arguments as before, 
(namely Lemma~\ref{lem:lemma2}), there exists at least one solution $z_n$ to the minimization problem (\ref{eq:minRay1}) and $u_n :=\frac{u_{n-1}+z_n}{\|u_{n-1} + z_n\|}$ satisfies
 $\lambda_n = a(u_n, _n) \leq \lambda_{n-1} < \lambda_{\Sigma}$. Thus, all the iterations of the PRaGA are well-defined, and the sequence $(\lambda_n)_{n\in\N}$ is non-increasing. 
\end{proof}

\subsubsection{Pure Residual Greedy Algorithm}\label{sec:residual}

The \itshape Pure Residual Greedy Algorithm \normalfont (PReGA) we propose is based on the use of a residual for problem (\ref{eq:eigenvalue}). 

\medskip

\bfseries Pure Residual Greedy Algorithm (PReGA): \normalfont

\begin{itemize}
\item \bfseries Initialization: \normalfont choose an initial guess $u_0\in V$ such that $\|u_0\| = 1$ and let $\lambda_0:= a(u_0,u_0)$;
\item \bfseries Iterate on $n\geq 1$:  \normalfont find $z_n\in \Sigma$ such that
\begin{equation}\label{eq:algo3}
z_n \in \mathop{\mbox{argmin}}_{z\in \Sigma} \frac{1}{2}\|u_{n-1} + z\|_a^2 - (\lambda_{n-1} + \nu)\langle u_{n-1}, z\rangle,
\end{equation}
and set $u_n := \frac{u_{n-1} + z_n}{\|u_{n-1} + z_n\|}$ and $\lambda_n := a(u_n, u_n)$. 
\end{itemize}

\medskip

The denomination \itshape Residual \normalfont can be justified as follows: it is easy to check that for all $n\in \N^*$, 
the minimization problem (\ref{eq:algo3}) is equivalent to the minimization problem
\begin{equation}\label{eq:resform}
 \mbox{find }z_n\in \Sigma\mbox{ such that }z_n\in \mathop{\mbox{argmin}}_{z\in \Sigma} \frac{1}{2}\|R_{n-1} -z \|_a^2,
\end{equation}
where $R_{n-1}\in V$ is the Riesz representant in $V$ of the linear form $l_{n-1}: v\in V \mapsto \lambda_{n-1}\langle u_{n-1}, v \rangle - a(u_{n-1}, v)$. In other words, $R_{n-1}$ is the unique element in $V$ such that
$$
\forall v\in V, \quad \langle R_{n-1} , v \rangle_a = \lambda_{n-1}\langle u_{n-1}, v \rangle - a(u_{n-1}, v).
$$
The linear form $l_{n-1}$ can indeed be seen as a residual for (\ref{eq:eigenvalue}) since $l_{n-1} = 0$ if and only if $\lambda_{n-1}$ is an eigenvalue of $a(\cdot, \cdot)$ and $u_{n-1}$ an associated 
$H$-normalized eigenvector. 

\medskip

Let us point out that, in order to carry out the PReGA in practice, one needs to know the value of a constant $\nu$ ensuring (HA), whereas this is not needed for the PRaGA, neither for the 
algorithm (PEGA) introduced in~\cite{Chinesta-QC} and considered in the next section. We will discuss in more details about 
the practical implementation of these three algorithms in the case when $\Sigma$ is the set of rank-1 tensor product functions in Section~\ref{sec:resnum}.

\medskip

\begin{lemma}\label{lem:Redef}
Let $V$ and $H$ be separable Hilbert spaces such that the embedding $V\hookrightarrow H$ is dense, $\Sigma$ a dictionary of $V$ and $a: V\times V\to \R$ a symmetric 
continuous bilinear form satisfying (HA). Then, all the iterations of the PReGA algorithm are well-defined in the sense that for all $n\in\N^*$, there exists at least one solution to the minimization 
problem (\ref{eq:algo3}).   
\end{lemma}

\begin{proof}
Lemma~\ref{lem:lemconvex} implies that for all $n\in\N^*$, there always exists at least one solution to the minimization problem (\ref{eq:algo3}), since for all $n\in\N^*$, the functional 
$$
\cE_n: \left\{ \begin{array}{ccc}
                V & \to & \R\\
                v & \mapsto & \cE_n(v):= \frac{1}{2}\|u_{n-1} + v\|_a^2 - (\lambda_{n-1} + \nu)\langle u_{n-1}, v\rangle,\\
               \end{array}
\right .
$$
satisfies (HE1) and (HE2). Thus, all the iterations of the  PReGA are well-defined. 
\end{proof}

\subsubsection{Pure Explicit Greedy Algorithm}\label{sec:PEGA}

The above two algorithms are new, at least to our knowledge. 
In this section, we describe the algorithm already proposed in~\cite{Chinesta-QC}, which we call in the rest of the article 
the \itshape Pure Explicit Greedy Algorithm \normalfont (PEGA). 

Unlike the above two algorithms, the PEGA is not defined for general dictionaries $\Sigma$ satisfying (H$\Sigma$1), (H$\Sigma$2) and (H$\Sigma$3). 
We need to assume in addition that $\Sigma$ is an embedded manifold in $V$. In this case, for all $z\in \Sigma$, we denote by $T_\Sigma(z)$ the tangent subspace to $\Sigma$ at the point $z$ in $V$. 

\medskip

Let us point out that, if $\Sigma$ is an embedded manifold in $V$, for all $n\in \N^*$, the Euler equations associated to the minimization problems
 (\ref{eq:minRay1}) and (\ref{eq:algo3}) respectively read: 
\begin{equation}\label{eq:Eulerrayleigh1}
\forall \delta z\in T_\Sigma(z_n), \quad a\left( u_{n-1} + z_n, \delta z\right) = \lambda_n \langle u_{n-1} + z_n, \delta z\rangle,
\end{equation}
and
\begin{equation}\label{eq:Eulerres1}
\forall \delta z \in T_\Sigma(z_n), \quad a\left( u_{n-1} + z_n, \delta z\right) + \nu \langle z_n, \delta z\rangle  = \lambda_{n-1} \langle u_{n-1} , \delta z\rangle.
\end{equation} 

\medskip

The PEGA consists in solving at each iteration $n\in\N^*$ of the greedy algorithm the following equation, which is of a similar form as the Euler equations (\ref{eq:Eulerrayleigh1}) and 
(\ref{eq:Eulerres1}) above, 
\begin{equation}\label{eq:EulerExp1}
\forall \delta z \in T_\Sigma(z_n), \quad  a\left( u_{n-1} + z_n , \delta z\right)  = \lambda_{n-1} \langle u_{n-1} + z_n, \delta z \rangle.
\end{equation}
More precisely, the PEGA algorithm reads:

\medskip

\bfseries Pure Explicit Greedy Algorithm (PEGA): \normalfont

\medskip 

\begin{itemize}
 \item \bfseries Initialization: \normalfont choose an initial guess $u_0\in V$ such that $\|u_0\| = 1$ and let $\lambda_0:= a(u_0,u_0)$;
\item \bfseries Iterate for $n\geq 1$: \normalfont find $z_n\in \Sigma$ such that
\begin{equation}\label{eq:Eulerexp}
\forall \delta z \in T_\Sigma(z_n), \quad  a\left( u_{n-1} + z_n , \delta z\right) - \lambda_{n-1} \langle u_{n-1} + z_n, \delta z \rangle = 0,
\end{equation}
and set $u_n := \frac{u_{n-1} + z_n}{\|u_{n-1} + z_n\|}$ and $\lambda_n := a(u_n, u_n)$. 
\end{itemize}

Notice that (\ref{eq:Eulerexp}) is very similar to (\ref{eq:Eulerrayleigh1}) except that $\lambda_{n-1}$ is used instead of $\lambda_n$. 
It can be seen as an \itshape explicit \normalfont version of the PRaGA, hence the name \itshape Pure Explicit 
Greedy Algorithm \normalfont.

\medskip

Note that it is not clear whether there always exists a solution $z_n$ to (\ref{eq:Eulerexp}), since (\ref{eq:Eulerexp}) does not derive from a minimization problem, unlike the other two 
algorithms. We were unable to prove convergence results for the PEGA.

\subsubsection{Orthogonal algorithms}\label{sec:orth}

We introduce here slightly modified versions of the PRaGA, PReGA and PEGA, inspired from the OGA for convex minimization 
problems (see Section~\ref{sec:convex}). 

\medskip

\bfseries Orthogonal (Rayleigh, Residual or Explicit) Greedy Algorithm (ORaGA, OReGA and OEGA): \normalfont

\begin{itemize}
 \item \bfseries Initialization: \normalfont choose an initial guess $u_0\in V$ such that $\|u_0\| = 1$ and let $\lambda_0:= a(u_0,u_0)$. For the ORaGA, we need to assume that 
$\lambda_0:= a(u_0,u_0) < \lambda_\Sigma$. 
\item \bfseries Iterate on $n\geq 1$: \normalfont 
\begin{itemize}
\item for the ORaGA: find $z_n\in \Sigma$ satisfying (\ref{eq:minRay1});
\item for the OReGA: find $z_n\in \Sigma$ satisfying (\ref{eq:algo3});  
\item for the OEGA: find $z_n\in \Sigma$ satisfying (\ref{eq:Eulerexp});  
\end{itemize}
find $\left(c_0^{(n)}, \ldots, c_n^{(n)} \right) \in \R^{n+1}$ such that
\begin{equation}\label{eq:orthopt}
\left( c_0^{(n)}, \ldots, c_n^{(n)}\right) \in \mathop{\rm argmin}_{(c_0, \ldots, c_n) \in \R^{n+1}} \cJ\left(c_0 u_0 + c_1 z_1+ \cdots + c_nz_n\right),
\end{equation}
and set  $u_n:= \frac{c_0^{(n)} u_0 + c_1^{(n)} z_1 + \cdots + c_n^{(n)} z_n}{\|c_0^{(n)} u_0 + c_1^{(n)} z_1 + \cdots + c_n^{(n)} z_n\|}$; 
if $\left\langle u_{n-1}, u_n \right\rangle \leq 0$, set $u_n:=-u_n$; set $\lambda_n:= a(u_n, u_n)$.
\end{itemize}

Let us point out that the original algorithm proposed in~\cite{Chinesta-QC} is the OEGA. Besides, for the three algorithms and all $n\in\N^*$, there always exists at least one solution to 
the minimization problems (\ref{eq:orthopt}).

The orthogonal versions of the greedy algorithms can be easily implemented from the pure versions: at any iteration $n\in\N^*$, only an additional step is performed, which 
consists in choosing an approximate eigenvector $u_n$ as a linear combination of the elements $u_0, z_1, \ldots, z_n$ minimizing the Rayleigh quotient 
associated to the bilinear form $a(\cdot, \cdot)$. Since $u_n$ is called to be the approximation of an eigenvector associated to the lowest eigenvalue of $a(\cdot, \cdot)$, which is a minimizer 
of the Rayleigh quotient on the Hilbert space $V$, this additional step is a natural extension of the OGA.   

\subsection{Convergence results}\label{sec:main}

\subsubsection{The infinite-dimensional case}

\begin{theorem}\label{th:main}
Let $V$ and $H$ be separable Hilbert spaces satisfying (HV), $\Sigma$ a dictionary of $V$ and $a: V\times V\to \R$ a symmetric 
continuous bilinear form satisfying (HA). The following properties hold for 
the PRaGA, ORaGA, PReGA and OReGA:
\begin{enumerate}
\item All the iterations of the algorithms are well-defined.
\item The sequence $(\lambda_n)_{n\in\N}$ is non-increasing and converges towards a limit $\lambda$ 
which is an eigenvalue of $a(\cdot,\cdot)$ for the scalar product $\langle \cdot, \cdot\rangle$.
\item The sequence $(u_n)_{n\in\N}$ is bounded in $V$ and any subsequence of $(u_n)_{n\in\N}$ which weakly converges in 
$V$ also  strongly converges in $V$ towards an $H$-normalized eigenvector associated with $\lambda$. This implies in particular that
$$
d_a(u_n, F_\lambda):= \mathop{\inf}_{w\in F_\lambda} \|w-u_n\|_a \mathop{\longrightarrow}_{n\to\infty} 0,
$$   
where $F_\lambda$ denotes the set of the $H$-normalized eigenvectors of $a(\cdot, \cdot)$ associated with $\lambda$.
\item If $\lambda$ is a simple eigenvalue, then there exists an $H$-normalized eigenvector $w_\lambda$ associated with $\lambda$ such that the whole sequence 
$(u_n)_{n\in\N}$ converges to $w_\lambda$ strongly in $V$.
\end{enumerate}
\end{theorem}
It may happen that $\lambda > \mu_1$, if the initial guess $u_0$ is not properly chosen. 
This point is discussed in Section~\ref{sec:initguess}. If $\lambda$ is degenerate, it is not clear whether the whole sequence $(u_n)_{n\in\N}$ converges.
 We will see however in Section~\ref{sec:finited} that it is always the case in finite dimension, at least for the pure versions of these algorithms.

The proof of Theorem~\ref{th:main} is given in Section~\ref{sec:proofRayleigh} for the PRaGA, in Section~\ref{sec:proofResidual} for the PReGA, and in Section~\ref{sec:prooforth} 
for their orthogonal versions.

\medskip

In addition, for the PReGA and the OReGA, we can prove similar convergence results without assuming that the Hilbert space $V$ is compactly embedded in $H$, provided that the self-adjoint operator 
$A$ associated with the quadratic form $a(\cdot, \cdot)$ has at least one eigenvalue below the minimum of its essential spectrum.

\begin{proposition}\label{prop:noHV}
Let $V$ and $H$ be separable Hilbert spaces such that the embedding $V\hookrightarrow H$ is dense (but not necessarily compact), $\Sigma$ a dictionary of $V$, $a: V\times V\to \R$ a symmetric 
continuous bilinear form satisfying (HA), and $A$ the self-adjoint operator on $H$ associated to $a(\cdot, \cdot)$. Let 
us assume also that $\min \sigma(A) < \min \sigma_{\rm ess}(A)$, where $\sigma(A)$ and $\sigma_{\rm ess}(A)$ respectively denote the spectrum and the essential spectrum of $A$, and that the initial 
guess $u_0$ satisfies $\min \sigma(A) \leq \lambda_0:=a(u_0,u_0) < \min \sigma_{\rm ess}(A)$. Then, the following properties hold for the PReGA and the OReGA:
\begin{enumerate}
\item All the iterations of the algorithms are well-defined.
\item The sequence $(\lambda_n)_{n\in\N}$ is non-increasing and converges towards a limit $\lambda$ 
which is an eigenvalue of $a(\cdot,\cdot)$ for the scalar product $\langle \cdot, \cdot\rangle$ such that $\lambda < \min \sigma_{\rm ess}(A)$.
\item The sequence $(u_n)_{n\in\N}$ is bounded in $V$ and any subsequence of $(u_n)_{n\in\N}$ which weakly converges in 
$V$ also  strongly converges in $V$ towards an $H$-normalized eigenvector associated with $\lambda$. This implies in particular that
$$
d_a(u_n, F_\lambda):= \mathop{\inf}_{w\in F_\lambda} \|w-u_n\|_a \mathop{\longrightarrow}_{n\to\infty} 0,
$$   
where $F_\lambda$ denotes the set of $H$-normalized eigenvectors of $a(\cdot, \cdot)$ associated with $\lambda$.
\item If $\lambda$ is a simple eigenvalue, then there exists an $H$-normalized eigenvector $w_\lambda$ associated with $\lambda$ such that the whole sequence 
$(u_n)_{n\in\N}$ converges to $w_\lambda$ strongly in $V$.
\end{enumerate}
\end{proposition}
 
\begin{remark} The above proposition shows that the PReGA or OReGA can be used to solve electronic structure calculation problems (at least in principle). 
Indeed, let us consider a molecular system composed of $d$ electrons and $M$ nuclei, with electric charges $(Z_k)_{1\leq k \leq M} \in \left(\N^*\right)^M$, 
and positions $\left(R_k\right)_{1\leq k \leq M} \in \left(\R^3\right)^M$. The electronic ground state is the lowest eigenstate of (\ref{eq:eigenvalue}) with 
$H := \bigwedge_{i=1}^d L^2(\R^3)$ the space of square-integrable antisymmetric functions on $\left(\R^{3}\right)^d$, $V:= \bigwedge_{i=1}^d H^1(\R^3)$, and
$$
\forall v,w \in V, \quad a(v,w):= \frac{1}{2}\int_{\R^{3d}}\nabla v \cdot \nabla w + \int_{\R^{3d}} W vw,
$$
where 
$$
W(x_1, \cdots, x_d) := - \sum_{i=1}^d \sum_{k=1}^M \frac{Z_k}{|x_i - R_k|} + \sum_{1\leq i < j \leq d}.
$$
It is well-known that $a(\cdot, \cdot)$ satisfies assumption (HA). In addition, if the system is neutral or positively charged (i.e. if $\sum_{k=1}^M Z_k \geq d$), then the self-adjoint operator 
$A = -\frac 1 2 \Delta + W$ on $H$ has an infinite number of eigenvalues below the minimum of its essential spectrum. We denote by $\Sigma^\cS$ the set of 
the Slater determinants (up to a multiplicative constant), i.e.
$$
\Sigma^\cS:= \left\{ \psi(x) := c\;  \mbox{\rm det}\left( \phi_i(x_j)\right)_{1\leq i,j\leq d} \, |\; c\in \R, \; \forall 1\leq i,j\leq d, \; \phi_i \in H^1(\R^3), \; \int_{\R^3} \phi_i \phi_j  = \delta_{ij} \right\}.
$$
Then, the embedding $V\hookrightarrow H$ is dense and $\Sigma^{\cS}$ is a dictionary of $V$ in the sense of Definition~\ref{def:dictionary}. Thus, the assumptions and results of Proposition~\ref{prop:noHV} 
hold and the PReGA and OReGA can be used in order to give an approximation of the lowest eigenvalue of $a(\cdot, \cdot)$ and an associated eigenvector. How to implement efficiently such an algorithm in practice will be the 
object of a forthcoming article.  \end{remark}
 
The proof of Proposition~\ref{prop:noHV} is given in Section~\ref{sec:proofnoHV} for the PReGA, and in Section~\ref{sec:prooforth} for the OReGA.

\subsubsection{The finite-dimensional case}\label{sec:finited}

From now on, for any differentiable function $f: V \to \R$, and all $v_0\in V$, we denote by $f'(v_0)$ the derivative of the function $f$ at the point $v_0\in V$. 
More precisely, $f'(v_0)\in V'$ 
is the unique continuous linear form on $V$ such that for all $v\in V$, 
$$
f(v) = f(v_0) + \langle f'(v_0), v \rangle_{V', V} + r(v), \mbox{ with }\mathop{\lim}_{\|v\|_a \to 0} \frac{r(v)}{\|v\|_a} = 0.
$$

Besides, we define the injective norm on $V'$ associated to $\Sigma$ as follows: 
$$
\forall l\in V', \; \|l\|_* = \mathop{\sup}_{z\in \Sigma^*} \frac{\langle l, z\rangle_{V',V}}{\|z\|_a}.
$$

In the rest of this section, we assume that $V$, hence $H$ (since the embedding $V\hookrightarrow H$ is dense), are finite dimensional vector spaces. The convergence results below heavily rely on the \L ojasiewicz inequality~\cite{Lojas} and the ideas presented in~\cite{Levitt} for the proof of 
convergence of gradient-based algorithms for the Hartree-Fock equations.

The \L ojasiewicz inequality~\cite{Lojas} reads as follows: 

\begin{lemma}\label{lem:Loja}
Let $\Omega$ be an open subset of the finite-dimensional Euclidean space $V$, and $f$ an analytic real-valued function defined on $\Omega$. 
Then, for each $v_0 \in \Omega$, there is a neighborhood $U\subset \Omega$ of $v_0$ and two constants $K \in \R_+$ and $\theta\in (0,1/2]$ such that for all $v\in U$, 
\begin{equation}\label{eq:Lojainit}
|f(v) - f(v_0)|^{1-\theta} \leq K \|f'(v)\|_*.
\end{equation}
\end{lemma}
This inequality can be understood in this way: it can be easily proved in the case when $v_0$ is not a critical point of $f$. When $v_0$ is a non-degenerate critical point, i.e. when the Hessian of $f$ at $v_0$ is invertible, 
then it is easy to see that $\theta$ can be chosen to be equal to $\frac{1}{2}$ by using a simple Taylor expansion. Moreover, when $v_0$ is a degenerate critical point of $f$, the analyticity assumption 
ensures that there exists $N\in\N^*$ such that the $N^{th}$-order derivatives cannot vanish simultaneously, and the exponent $\theta$ can be chosen to be equal to $\frac{1}{N}$. 

\medskip

Before stating our main result in finite dimension, we prove a useful lemma.
\begin{lemma}\label{lem:Lojaadap}
Let $V$ and $H$ be finite-dimensional Euclidean spaces, $\Omega:= \{ v\in V, \; 1/2 < \|v\| < 3/2 \}$, $\lambda$ be an eigenvalue of the bilinear form $a(\cdot, \cdot)$ 
and $F_\lambda$ be the set of the $H$-normalized eigenvectors of $a(\cdot, \cdot)$ associated to $\lambda$. 
Then, $\cJ: \Omega \to \R$ is analytic, 
and there exists $K  \in \R_+$, $\theta \in (0,1/2]$ and $\varepsilon >0$ such that
\begin{equation}\label{eq:Lojanew}
\mbox{for all } v\in \Omega \mbox{ such that } d(v,F_\lambda):= \mathop{\inf}_{w\in F_\lambda} \|v-w\| \leq \varepsilon, \quad \left|\cJ(v) - \lambda\right|^{1-\theta} \leq K \|\cJ'(v)\|_*.
\end{equation}
\end{lemma}

\begin{proof}
The functional $\cJ:\Omega \to \R$ is analytic as a composition of analytic functions. Thus, from (\ref{eq:Lojainit}), 
for all $w\in F_\lambda$, there exists $\varepsilon_w>0$, $K_w \in \R_+$ and $\theta_w\in (0,1/2]$ such that
\begin{equation}\label{eq:lojw}
\forall v\in B(w,\varepsilon_w), \quad |\cJ(v) - \lambda|^{1-\theta_w} \leq K_w \|\cJ'(v)\|_*,
\end{equation}
where $B(w, \varepsilon_w):= \left\{ v\in V, \; \|v-w\|\leq \varepsilon_w\right\}$. Besides, for all $w\in F_\lambda$, we can choose $\varepsilon_w$ small enough so that 
$B(w, \varepsilon_w) \subset \Omega$. The family $(B(w, \varepsilon_w))_{w\in F_\lambda}$ forms a cover of open sets of $F_\lambda$. 
Since $F_\lambda$ is a compact subset of $V$ (it is a closed bounded subset of a finite-dimensional space), we can extract a finite subcover from the family $(B(w, \varepsilon_w))_{w\in F_\lambda}$, 
from which we deduce the existence of constants $\varepsilon>0$, $K >0$ and $\theta \in (0,1/2]$ such that
$$
\mbox{for all } v\in \Omega \mbox{ such that } d(v, F_\lambda) \leq \varepsilon, \quad \left|\cJ(v) - \lambda\right|^{1-\theta} \leq K \|\cJ'(v)\|_*.
$$
Hence the result.
\end{proof}

The proof of the following Theorem is given in Section~\ref{sec:prooffinited}.
\begin{theorem}\label{prop:finitedim}
Let $V$ and $H$ be finite dimensional Euclidian spaces and $a:V\times V\to \R$ be a symmetric bilinear form. The following properties hold for both PRaGA and PReGA:
\begin{enumerate}
\item the whole sequence $(u_n)_{n\in\N}$ strongly converges in $V$ to some $w_\lambda \in F_\lambda$;
\item the convergence rates are as follows, depending on the value of the parameter $\theta$ in (\ref{eq:Lojanew}):
\begin{itemize}
\item if $\theta = 1/2$, there exists $C\in \R_+$ and $0 < \sigma < 1$ such that for all $n\in \N$,
\begin{equation}\label{eq:rate1}
\|u_n -w_\lambda\|_a \leq C\sigma^n;
\end{equation}
\item if $\theta \in (0, 1/2)$, there exists $C \in \R_+$ such that for all $n\in \N^*$,
\begin{equation}\label{eq:rate2}
\|u_n -w_\lambda\|_a \leq C n^{-\frac{\theta}{1-2\theta}}.
\end{equation}
\end{itemize}
\end{enumerate}
\end{theorem}

\subsection{Discussion about the initial guess}

\subsubsection{Possible choice of initial guess}\label{sec:initguess}

We present here a generic procedure to choose an initial guess $u_0\in V$ satisfying $\|u_0\| = 1$ and $a(u_0,u_0) \leq \lambda_\Sigma$:

\medskip

\bfseries Choice of an initial guess: \normalfont

\begin{itemize}
\item \bfseries Initialization: \normalfont choose $z_0\in \Sigma$ such that
\begin{equation}\label{eq:mininit}
z_0 \in \mathop{\mbox{argmin}}_{z\in \Sigma} \cJ(z),
\end{equation}
and set $u_0  := \frac{z_0}{\|z_0\|}$.
\end{itemize}

From Lemma~\ref{lem:lemma2}, (\ref{eq:mininit}) always has at least one solution and it is straightforward to see that $\|u_0\| = 1$ and $a(u_0,u_0) = \lambda_\Sigma$. 
In all the numerical tests presented in Section~\ref{sec:numtest}, our initial guess is chosen according to this procedure. 

\subsubsection{Special case of the PRaGA}\label{sec:initRay}

Let us recall that in the case of the PRaGA, we required that the initial guess $u_0$ of the algorithm satisfies $a(u_0, u_0) < \lambda_\Sigma$, whereas the above procedure generates an 
initial guess $u_0$ with $a(u_0,u_0) = \lambda_\Sigma$. Let us comment on this condition. We distinguish here two different cases:
\begin{itemize}
 \item If the element $u_0$ computed with the procedure presented in Section~\ref{sec:initguess}, is an eigenvector of $a(\cdot, \cdot)$ associated to the eigenvalue $\lambda_0$, then 
from Lemma~\ref{lem:leminit}, $\cJ(u_0 + z) \geq \cJ(u_0)$ for all $z\in \Sigma$. We exclude this case from now on in all the rest of the article.
Let us point out though that this case happens only in very particular situations. 
Indeed, it can be proved that if we consider the prototypical example presented in Section~\ref{sec:example} with 
$b = 1$, $W$ a H\"older-continuous function (this assumption can be weakened) and $\Sigma = \Sigma^\otimes$ defined by (\ref{eq:rank1}), then an element $z\in \Sigma$ is an eigenvector associated to the bilinear form 
$a(\cdot, \cdot)$ defined by (\ref{eq:defaex}) if and only if the potential $W$ can be written as a sum of one-body potentials of the form
$$
W(x_1, \ldots, x_d) = W_1(x_1) + \cdots + W_d(x_d). 
$$
We provide a proof of this result in the Appendix.
\item If $u_0$ is not an eigenvector of $a(\cdot, \cdot)$ associated to the eigenvalue $\lambda_0$, then from Lemma~\ref{lem:leminit}, 
there exists some $z \in \Sigma$ such that $\cJ(u_0 + z) < \cJ(u_0)$. Thus, up to taking $u_0:= u_0 +z$ as the new initial guess, 
we have that $\lambda_0:= a(u_0, u_0) < \lambda_\Sigma$. 
\end{itemize}

\subsubsection{Convergence towards the lowest eigenstate}

As mentioned above, the greedy algorithms may not converge towards the {\em lowest} eigenvalue of the bilinear form $a(\cdot, \cdot)$ 
depending on the choice of the initial guess $u_0$. Of course, if $u_0$ is chosen so that 
$\lambda_0 = a(u_0,u_0) < \mu^*_2:= \mathop{\inf}_{j\in\N^*}\left\{ \mu_j\; | \; \mu_j > \mu_1\right\}$, then the sequences
 $(\lambda_n)_{n\in\N}$ generated by the greedy algorithms automatically converge to $\mu_1$. However, the construction of such an initial guess $u_0$
 in the general case is not obvious.  

\medskip

One might hope that using the procedure presented in Section~\ref{sec:initguess} to choose the initial guess $u_0$ would be sufficient to ensure that the greedy algorithms 
converge to $\mu_1$. Unfortunately, this is not the case, as shown in Example~\ref{ex:ex2}.
However, we believe that this only happens in pathological situations, and that, in most practical cases, 
the eigenvalue approximated by a greedy algorithm using this procedure to determine the initial guess is indeed $\mu_1$.

\section{Numerical implementation}\label{sec:resnum}

In this section, we present how the above algorithms, and the one proposed in~\cite{Chinesta-QC}, can be implemented in practice in the case when 
$\Sigma$ is the set of rank-1 tensor product functions of the form
(\ref{eq:rank1}): $\Sigma:= \Sigma^{\otimes}$. 

We consider here the case when $V$ and $H$ are Hilbert spaces of functions depending on $d$ variables $x_1, \;\ldots, \; x_d$, for some $d\in \N^*$, such that (HV) is satisfied. For all 
$1\leq j \leq d$, let $V_j$ be a Hilbert space of functions depending only on the variable $x_j$ such that the subset
\begin{equation}\label{eq:format}
\Sigma:= \left\{ r^{(1)} \otimes \cdots \otimes r^{(d)}\; | \; r^{(1)} \in V_1, \ldots, r^{(d)} \in V_d \right\}
\end{equation}
is a dictionary of $V$, according to Definition~\ref{def:dictionary}. For all $z_n\in \Sigma$ such that $z_n = r_n^{(1)}\otimes \cdots \otimes r_n^{(d)}$ with 
$\left( r_n^{(1)} , \ldots , r_n^{(d)} \right) \in V_1 \times \cdots \times V_d$, we define the tangent space to $\Sigma$ at $z_n$ as 
\begin{align*}
T_\Sigma(z_n) & :=  \left\{ \delta r^{(1)} \otimes r_n^{(2)} \otimes \cdots \otimes r_n^{(d)} + r_n^{(1)}\otimes \delta r^{(2)} \otimes \cdots \otimes r_n^{(d)} + \cdots + r_n^{(1)} \otimes r_n^{(2)} \otimes \cdots \otimes \delta r^{(d)} \; | \right. \\
&\left. \delta r^{(1)} \in V_1, \ldots, \delta r^{(d)} \in V_d\right\}. \\
\end{align*}

\subsection{Computation of the initial guess}\label{sec:firstit}

The initial guess $u_0\in V$ of all the algorithms is computed as follows: choose
$$
u_0:= z_0 = r_0^{(1)} \otimes \cdots \otimes r_0^{(d)} \in \mbox{\rm argmin}_{\left(r^{(1)}, \ldots, r^{(d)}\right)\in V_1 \times \cdots \times V_d} \cJ\left(r^{(1)} \otimes \cdots \otimes r^{(d)}\right),
$$
such that $\|u_0\| = \|z_0\| = 1$. 
To compute this initial guess in practice, we use the well-known Alternating Direction Method (ADM) (also called in the literature Alternating Least Square method
 in~\cite{Hackbusch, Schneider, Schneider2}, or fixed-point procedure in~\cite{Chinesta-QC, LBLM}): 
\begin{itemize}
 \item \bfseries Initialization: \normalfont choose $\left( s_0^{(1)}, \ldots, s_0^{(d)} \right)\in V_1 \times \cdots \times V_d$ 
such that $\left\| s_0^{(1)} \otimes \cdots \otimes s_0^{(d)} \right\| = 1$;
 \item \bfseries Iterate on $m=1, \ldots , m_{max}$: \normalfont 
\begin{itemize}
\item \bfseries Iterate on $j=1, \ldots,  d$: \normalfont choose $s_m^{(j)} \in V_j$ such that
\begin{equation}\label{eq:smj}
s_m^{(j)} \in \mathop{\mbox{\rm argmin}}_{s^{(j)}\in V_j} \cJ\left( s_m^{(1)}\otimes \cdots \otimes s_m^{(j-1)} \otimes s^{(j)} \otimes s_{m-1}^{(j+1)} \otimes \cdots \otimes s_{m-1}^{(d)}\right); 
\end{equation}
\end{itemize}
\item set $r_0^{(1)} \otimes \cdots \otimes r_0^{(d)} = s_m^{(1)} \otimes \cdots \otimes s_m^{(d)}$. 
\end{itemize}

It is observed that the ADM algorithm converges quite fast in practice. 
Actually, the resolution of (\ref{eq:smj}) amounts to computing the smallest eigenvalue and an associated eigenvector of a \itshape low-dimensional \normalfont eigenvalue problem, since 
$s_m^{(j)}$ is an eigenvector associated to the smallest eigenvalue of the bilinear form $a_{m,j}: V_j \times V_j \to \R$ with respect to the scalar product 
$\langle \cdot, \cdot \rangle_{m,j}: V_j \times V_j \to \R$, such that for all $v_1^{(j)}, v_2^{(j)} \in V_j$,
$$
a_{m,j}\left( v_1^{(j)}, v_2^{(j)}\right) = a\left( s_m^{(1)}\otimes \cdots \otimes s_m^{(j-1)} \otimes v_1^{(j)} \otimes s_{m-1}^{(j+1)} \otimes \cdots \otimes s_{m-1}^{(d)}, s_m^{(1)}\otimes \cdots \otimes s_m^{(j-1)} \otimes v_2^{(j)} \otimes s_{m-1}^{(j+1)} \otimes \cdots \otimes s_{m-1}^{(d)}\right), 
$$
and
$$
\left\langle v_1^{(j)}, v_2^{(j)}\right\rangle_{m,j} = \left\langle s_m^{(1)}\otimes \cdots \otimes s_m^{(j-1)} \otimes v_1^{(j)} \otimes s_{m-1}^{(j+1)} \otimes \cdots \otimes s_{m-1}^{(d)}, s_m^{(1)}\otimes \cdots \otimes s_m^{(j-1)} \otimes v_2^{(j)} \otimes s_{m-1}^{(j+1)} \otimes \cdots \otimes s_{m-1}^{(d)}\right\rangle. 
$$

\subsection{Implementation of the Pure Rayleigh Greedy Algorithm}\label{sec:rayleighit}

We now detail how the iterations of the PRaGA presented in Section~\ref{sec:rayleigh} are implemented in practice. The Euler equation 
associated to the minimization problem (\ref{eq:minRay1}) reads:  
\begin{equation}\label{eq:Eulerrayleigh}
\forall \delta z \in T_\Sigma(z_n), \quad  a\left( u_{n-1} + z_n , \delta z\right) - \lambda_n \langle u_{n-1} + z_n, \delta z \rangle = 0,
\end{equation}
where we recall that $\lambda_n:= a(u_n,u_n)$. 
We also use an ADM procedure to compute the tensor product $z_n=r_n^{(1)} \otimes \cdots \otimes r_n^{(d)}$, which reads as follows:
\begin{itemize}
 \item \bfseries Initialization: \normalfont choose $\left( s_0^{(1)}, \ldots, s_0^{(d)}\right)  \in V_1 \times \cdots \times V_d$;
\item \bfseries Iterate on $m = 1,  \ldots,  m_{max}$: \normalfont
\begin{itemize}
 \item \bfseries Iterate on $j = 1,  \ldots,  d$: \normalfont find $s_m^{(j)}\in V_j$ such that
\begin{equation}\label{eq:ALSalgo2}
s_m^{(j)} \in \mathop{\mbox{\rm argmin}}_{s^{(j)} \in V_j} \cJ \left( u_{n-1} + s_m^{(1)} \otimes \cdots \otimes s_m^{(j-1)} \otimes s^{(j)} \otimes s_{m-1}^{(j+1)} \otimes \cdots \otimes s_{m-1}^{(d)} \right);
\end{equation}
\end{itemize}
\item set $\left( r_n^{(1)}, \ldots, r_n^{(d)} \right)  = \left( s_m^{(1)}, \ldots, s_m^{(d)} \right)$. 
\end{itemize}
For $n\geq 1$, the minimization problems (\ref{eq:ALSalgo2}) are well-defined. 
Let us now detail a method for solving (\ref{eq:ALSalgo2}) in the discrete case, which seems to be new. For all $1\leq j \leq d$, let $N_j\in \N^*$ and let $\left(\phi_i^{(j)} \right)_{1\leq i \leq  {N_j}}$ 
be a Galerkin basis of a finite-dimensional subspace $V_{j, N_j}$ of $V_j$. The value of $N$ is fixed in the 
rest of this section. The discrete version of the algorithm reads:
\begin{itemize}
 \item \bfseries Initialization: \normalfont choose $\left( s_0^{(1)}, \ldots, s_0^{(d)}\right)  \in V_{1, N_1} \times \cdots \times V_{d, N_d}$;
\item \bfseries Iterate on $m = 1,  \ldots,  m_{max}$: \normalfont
\begin{itemize}
 \item \bfseries Iterate on $j = 1,  \ldots,  d$: \normalfont find $s_m^{(j)}\in V_{j,N_j}$ such that
\begin{equation}\label{eq:ALSalgo2d}
s_m^{(j)} \in \mathop{\mbox{\rm argmin}}_{s^{(j)} \in V_{j,N_j}} \cJ \left( u_{n-1} + s_m^{(1)} \otimes \cdots \otimes s_m^{(j-1)} \otimes s^{(j)} \otimes s_{m-1}^{(j+1)} \otimes \cdots \otimes s_{m-1}^{(d)} \right);
\end{equation}
\end{itemize}
\item set $\left( r_n^{(1)}, \ldots, r_n^{(d)} \right)  = \left( s_m^{(1)}, \ldots, s_m^{(d)} \right)$. 
\end{itemize}
We present below how (\ref{eq:ALSalgo2d}) is solved for a fixed value of $j\in \{1, \cdots, d\}$. To simplify the notation, we assume that all the $N_j$ are equal and denote by $N$ their common value. 
Denoting by $S = \left(S_i\right)_{1\leq i \leq N}\in\R^N$ the vector of the coordinates of the function $s^{(j)}$ in the basis $\left( \phi_i^{(j)} \right)_{1\leq i \leq N}$, so that
$$
s^{(j)} = \sum_{i=1}^N S_i \phi_i^{(j)},
$$
it holds
$$
\cJ \left( u_{n-1} + s_m^{(1)} \otimes \cdots \otimes s_m^{(j-1)} \otimes s^{(j)} \otimes s_{m-1}^{(j+1)} \otimes \cdots \otimes s_{m-1}^{(d)} \right) 
= \frac{S^T \cA S + 2 A^TS + \alpha}{S^T\cB S + 2 B^TS + 1},
$$
where the symmetric matrix $\cA\in \R^{N\times N}$, the positive definite symmetric matrix $\cB \in \R^{N\times N}$, the vectors $A,B \in \R^N$, and the real number $\alpha:= a(u_{n-1}, u_{n-1})$
 are independent of $S$. Making the change of variable 
 $T = \cB^{1/2}S + \cB^{-1/2}B$, we obtain 
$$
\cJ \left( u_{n-1} + s_m^{(1)} \otimes \cdots \otimes s_m^{(j-1)} \otimes s^{(j)} \otimes s_{m-1}^{(j+1)} \otimes \cdots \otimes s_{m-1}^{(d)} \right) 
= \cL(T) := \frac{T^T\cC T +2 C^T T + \gamma}{T^T T + \delta},
$$
where the symmetric matrix $\cC\in \R^{N\times N}$, the vector $C \in \R^N$ and the real numbers $\gamma\in \R$ and $\delta>0$ are independent of $T$.
Solving problem (\ref{eq:ALSalgo2d}) is therefore equivalent to solving
\begin{equation}\label{eq:discrete2}
\mbox{ find } T_m \in \R^N \mbox{ such that }
T_m \in \mathop{\mbox{\rm argmin}}_{T\in \R^N} \cL(T).
\end{equation}
An efficient method to solve (\ref{eq:discrete2}) is the following. Let us denote by $(\kappa_i)_{1\leq i \leq N}$ the eigenvalues of the matrix $\cC$ 
(counted with multiplicity) and let $(K_i)_{1\leq i \leq N }$ be an orthonormal family (for the Euclidean scalar product of $\R^N$) of associated eigenvectors. 
Let $(c_i)_{1\leq i\leq N}$ (resp. $(t_i)_{1\leq i\leq N}$) be the coordinates of the vector $C$ (resp. of the trial vector $T$)  in the basis $(K_i)_{1\leq i \leq N }$: 
$$
C = \sum_{i=1}^N c_i K_i, \quad T = \sum_{i=1}^N t_i K_i.
$$
We aim at finding $(t_{i,m})_{1\leq i \leq N}$ 
the coordinates of a vector $T_m$ solution of (\ref{eq:discrete2}) in the basis $(K_i)_{1\leq i \leq N }$. For any $T\in\R^N$, we have
$$
\cL(T) = \frac{\sum_{i=1}^N \kappa_i t_i^2 + 2 \sum_{i=1}^N c_i t_i +\gamma}{\sum_{i=1}^N t_i^2 + \delta}.
$$
Denoting by $\rho_m := \cL(T_m)\geq \mu_1$, the Euler equation associated with (\ref{eq:discrete2}) reads: 
$$
\forall 1\leq i \leq N, \quad \kappa_i t_{i,m} + c_i = \rho_m t_{i,m},
$$
so that
\begin{equation}\label{eq:coordti}
\forall 1\leq i \leq N, \quad t_{i,m} = \frac{c_i}{\rho_m - \kappa_i}.
\end{equation}
This implies that
$$
\cL(T_m) = \frac{\sum_{i=1}^N \kappa_i \frac{c_i^2}{(\rho_m - \kappa_i)^2} + 2 \sum_{i=1}^N \frac{c_i^2}{\rho_m - \kappa_i} + \gamma}{\sum_{i=1}^N \frac{c_i^2}{(\rho_m - \kappa_i)^2} + \delta}.
$$
Setting for all $\rho \in \R\setminus\{\kappa_i\}_{1\leq i \leq N} $, 
\begin{equation}\label{eq:Mrho}
\cM(\rho) = \frac{\sum_{i=1}^N \kappa_i \frac{c_i^2}{(\rho - \kappa_i)^2} + 2 \sum_{i=1}^N \frac{c_i^2}{\rho - \kappa_i} + \gamma}{\sum_{i=1}^N \frac{c_i^2}{(\rho - \kappa_i)^2} + \delta},
\end{equation}
it holds that 
$$
\rho_m = \cL(T_m) = \cM(\rho_m) \leq \mathop{\inf}_{\rho \in \R\setminus \{\kappa_i\}_{1\leq i \leq N}} \cM(\rho) = \mathop{\inf}_{\rho \in \R\setminus \{\kappa_i\}_{1\leq i \leq N}} \cL(T(\rho)),
$$
where $T(\rho) = \sum_{i=1}^N t_i(\rho) K_i$ with $t_i(\rho) = \frac{c_i}{\rho - \kappa_i}$ for all $1\leq i \leq N$. Thus,
\begin{equation} \label{eq:mu}
\rho_m = \mathop{\mbox{\rm argmin}}_{\rho \in \R\setminus \{\kappa_i\}_{1\leq i \leq N}} \cM(\rho). 
\end{equation}
The Euler equation associated with the one-dimensional minimization problem (\ref{eq:mu}) reads, after some algebraic manipulations,
$$
\rho_m \delta = \sum_{i=1}^N \frac{ c_i^2 }{\rho_m - \kappa_i} + \gamma. 
$$
Denoting by $f:\rho \in \R \setminus \{\kappa_i\}_{1\leq i \leq N} \mapsto \sum_{i=1}^N \frac{c_i^2}{\rho - \kappa_i} + \gamma$, we have the following lemma:
\begin{lemma}
Let $T_m$ be a solution to (\ref{eq:discrete2}). The real number $\rho_m := \cL(T_m)$ is the smallest solution to the equation
\begin{equation}\label{eq:rho}
\mbox{ find }\rho \in \R \setminus \{\kappa_i\}_{1\leq i \leq N} \mbox{ such that }\rho \delta = f(\rho).
\end{equation}
\end{lemma}

\begin{proof}
The calculations detailed above show that $\rho_m$ is a solution of (\ref{eq:rho}). On the other hand, for all $\rho \in \R$ satisfying (\ref{eq:rho}), 
it can be easily seen after some algebraic 
manipulations that $\rho = \cM(\rho) = \cL(T(\rho))$. Thus, since $\rho_m$ is solution to (\ref{eq:mu}), in particular, for all $\rho\in \R$ solution of (\ref{eq:rho}), we have
$$
\rho_m  = \cL(T(\rho_m)) = \cM(\rho_m)\leq \cM(\rho) = \rho = \cL(T(\rho)).
$$
\end{proof}

For all $1\leq i \leq N$, $f(\kappa_i^-) = -\infty$, $f(\kappa_i^+) = +\infty$, $f(-\infty)= f(+\infty) = \gamma$ and the function $f$ is decreasing on each interval $(\kappa_i, \kappa_{i+1})$ 
(with the convention $\kappa_0 = -\infty$ and $\kappa_{N+1} = +\infty$). Thus, equation (\ref{eq:rho}) has exactly one solution in each interval $(\kappa_i, \kappa_{i+1})$. Thus, $\rho_m$ is the unique 
solution of (\ref{eq:rho}) lying in the interval $(-\infty, \kappa_1)$ (see Figure~1).

 \begin{figure}\label{fig:figNewton}
   \centering
   \input{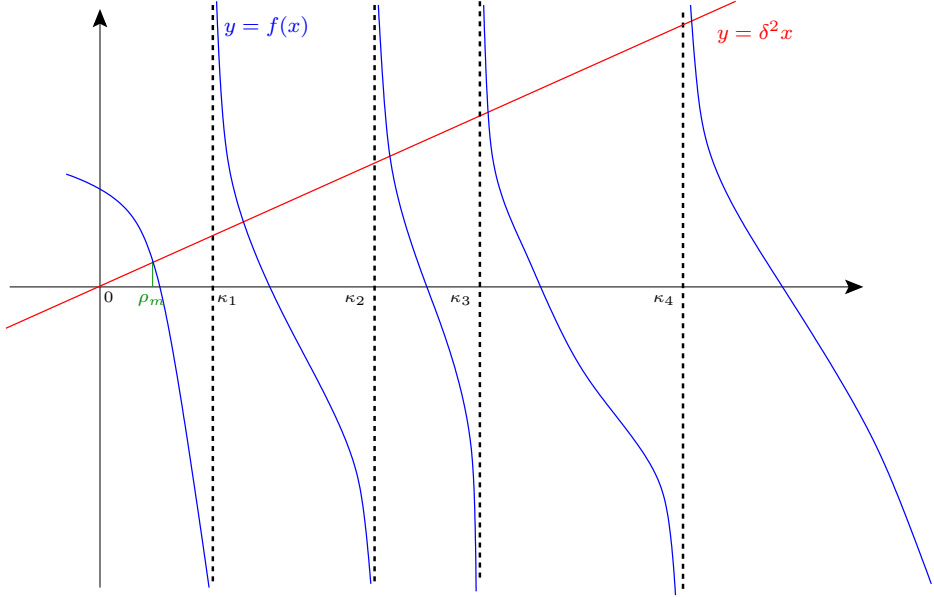}

   \caption{Solutions of equation (\ref{eq:rho}).}
   \label{fig:fig2}
 \end{figure}

We use a standard Newton algorithm to solve equation (\ref{eq:rho}).
 The coordinate of a vector $T_m$ solution of (\ref{eq:discrete2}) are then determined using (\ref{eq:coordti}). Thus, solving (\ref{eq:ALSalgo2d}) amounts to fully diagonalizing 
the low-dimensional $N\times N = N_j\times N_j$ matrix $\mathcal{C}$. 

\medskip

Let us point out that problems (\ref{eq:smj}) and (\ref{eq:ALSalgo2}) are of different nature: in particular, (\ref{eq:smj}) is an eigenvalue problem whereas (\ref{eq:ALSalgo2}) is not. 
In the discrete setting, the strategy presented in this section for the resolution of (\ref{eq:ALSalgo2}) could also be applied to the resolution of (\ref{eq:smj}); 
however, since it requires the full diagonalization of matrices of sizes $N_j\times N_j$, it is more expensive from a computational point of view than standard algorithms dedicated to the computation of the smallest 
eigenvalue of a matrix, which can be used for the resolution of (\ref{eq:smj}).

\subsection{Implementation of the Pure Residual Greedy Algorithm}

The Euler equation associated to the minimization problem (\ref{eq:algo3}) reads:  
$$
\forall \delta z \in T_\Sigma(z_n), \quad \langle u_{n-1} + z_n , \delta z\rangle_a  - (\lambda_{n-1} +\nu) \langle u_{n-1}, \delta z \rangle = 0.
$$
This equation is solved using again an ADM procedure, which reads as follows: 
\begin{itemize}
 \item \bfseries Initialization: \normalfont choose $\left( s_0^{(1)} , \ldots, s_0^{(d)} \right) \in V_1 \times \cdots \times V_d$; 
\item \bfseries Iterate on $m=1,  \ldots,  m_{max}$: \normalfont
\begin{itemize}
 \item \bfseries Iterate on $j = 1,  \ldots,  d$: \normalfont find $s_m^{(j)} \in V_j$ such that for all $\delta s^{(j)} \in V_j$,
\begin{equation}\label{eq:residualit}
\left\langle u_{n-1} + z_m^{(j)}, \delta z_m^{(j)} \right\rangle_a -( \lambda_{n-1} + \nu)\left \langle u_{n-1}, \delta z_m^{(j)} \right \rangle=0,
\end{equation}
where 
$$
z_m^{(j)} = s_m^{(1)} \otimes \cdots \otimes s_m^{(j-1)} \otimes s_m^{(j)} \otimes s_{m-1}^{(j+1)} \otimes \cdots \otimes s_{m-1}^{(d)}
$$
and
$$
\delta z_m^{(j)}  = s_m^{(1)} \otimes \cdots \otimes s_m^{(j-1)} \otimes \delta s^{(j)} \otimes s_{m-1}^{(j+1)} \otimes \cdots \otimes s_{m-1}^{(d)};
$$
\end{itemize}
\item set $\left( r_n^{(1)}, \ldots, r_n^{(d)} \right) = \left( s_m^{(1)}, \ldots , s_m^{(d)} \right)$.  
\end{itemize}

In our numerical experiments, we observed that this algorithm rapidly converges to a fixed point. Let us point out that using the same space discretization as in 
Section~\ref{sec:rayleighit}, namely a Galerkin basis of $N_j$ functions for all $1\leq j \leq d$, the resolution of (\ref{eq:residualit}) only requires the inversion 
(and not the diagonalization) of low-dimensional $N_j\times N_j$ matrices.

\subsection{Implementation of the Pure Explicit Greedy Algorithm}

At each iteration of this algorithm, equation (\ref{eq:Eulerexp}) is also solved using an ADM procedure, which reads:
\begin{itemize}
 \item \bfseries Initialization: \normalfont choose $\left( s_0^{(1)}, \ldots, s_0^{(d)} \right) \in V_1 \times \cdots \times V_d$ and set $m=1$; 
\item \bfseries Iterate on $m=1, \ldots , m_{max}$:\normalfont
\begin{itemize}
\item \bfseries Iterate on $j=1, \ldots, d$: \normalfont find $s_m^{(j)}\in V_j$ such that for all $\delta s^{(j)} \in V_j$,
$$
a\left( u_{n-1} + z_m^{(j)}, \delta z_m^{(j)} \right) - \lambda_{n-1}\langle u_{n-1} + z_m^{(j)}, \delta z_m^{(j)}\rangle = 0,
$$
where
$$
z_m^{(j)} = s_m^{(1)} \otimes \cdots \otimes s_m^{(j-1)}\otimes s_m^{(j)} \otimes s_{m-1}^{(j+1)} \otimes \cdots \otimes s_{m-1}^{(d)},
$$
and 
$$
\delta z_m^{(j)} = s_m^{(1)} \otimes \cdots \otimes s_m^{(j-1)}\otimes \delta s^{(j)} \otimes s_{m-1}^{(j+1)} \otimes \cdots \otimes s_{m-1}^{(d)};
$$
\end{itemize}
\item set $\left( r_n^{(1)}, \ldots, r_n^{(d)} \right)  = \left( s_m^{(1)} , \ldots , s_m^{(d)}\right)$. 
\end{itemize}
We observe numerically that this algorithm usually converges quite fast. However, we have noticed cases when this ADM procedure does not converge, which leads 
us to think that there may not always exist solutions $z_n\neq 0$ to (\ref{eq:Eulerexp}), even if $u_{n-1}$ is not an eigenvector associated to $a(\cdot, \cdot)$.

 \subsection{Implementation of the orthogonal versions of the greedy algorithms}

An equivalent formulation of (\ref{eq:orthopt}) is the following: find $\left(c_0^{(n)}, \ldots, c_n^{(n)} \right)\in \R^{n+1}$ such that
\begin{equation}\label{eq:orthv2}
\left( c_0^{(n)}, \ldots, c_n^{(n)}\right) \in \mathop{\rm argmin}_{(c_0, \ldots, c_n) \in \R^{n+1}, \; \|c_0 u_0 + c_1 z_1 + \cdots + c_nz_n\|^2 = 1} a\left(c_0 u_0 + c_1 z_1+ \cdots + c_nz_n\right).
\end{equation}
Actually, for all $0\leq k,l\leq n+1$, denoting by (using the abuse of notation $z_0 = u_0$):
\begin{align*}
\cB_{kl} & := \langle z_k, z_l\rangle\\
 \cA_{kl} & := a(z_k, z_l)\\
\end{align*}
and by $\cA:=\left(\cA_{kl}\right) \in \R^{(n+1)\times (n+1)}$ and $\cB:=\left( \cB_{kl} \right) \in \R^{(n+1)\times (n+1)}$, the 
vector $C^{(n)} = (c_0^{(n)}, \ldots , c_n^{(n)})\in \R^{n+1}$ is a solution of (\ref{eq:orthv2}) if and only if $C$ is an eigenvector associated to the smallest eigenvalue of 
the following generalized eigenvalue problem: 
$$
\left\{
\begin{array}{l}
\mbox{find }(\tau, C)\in \R\times \R^{n+1}\mbox{ such that } C^T\cB C = 1 \mbox{ and }\\
\cA C  =\tau \cB C, \\
\end{array}
\right .
$$
which is easy to solve in practice provided that $n$ remains small enough.

\section{Numerical results}\label{sec:numtest}

We present here some numerical results obtained with these algorithms (PRaGA, PReGA, PEGA and their orthogonal versions) on toy examples involving only 
two Hilbert spaces ($d=2$). We refer the reader to~\cite{Chinesta-QC} for numerical examples involving a larger number of variables. Section~\ref{sec:toy} presents basic numerical tests performed with small-dimensional matrices, which lead us to think that the greedy algorithms presented above converge in general towards the lowest eigenvalue of the bilinear form under consideration, 
except in pathological situations which are not likely to be encountered in practice. In Section~\ref{sec:plate}, the first buckling mode of a microstructured plate with defects 
is computed using these algorithms.

\subsection{A toy problem with matrices}\label{sec:toy}

In this simple example, we take $V = H = \R^{N_x\times N_y}$, $V_x = \R^{N_x}$ and $V_y = \R^{N_y}$ for some $N_x, N_y\in \N^*$ (here typically $N_x = N_y  = 51$). Let $D^{1x}, D^{2x} \in \R^{N_x\times N_x}$ and $D^{1y}, D^{2y} \in \R^{N_y\times N_y}$ be  (randomly chosen) symmetric definite positive matrices. We aim at computing the lowest eigenstate of the symmetric bilinear form 
$$
a(U,V) = \mbox{\rm Tr}\left( U^T (D^{1x} V D^{1y} + D^{2x} V D^{2y})\right),  
$$
or, in other words, of the symmetric fourth order tensor $A$ defined by
$$
\forall 1\leq i,k\leq N_x, \; 1\leq j,l \leq N_y, \quad A_{ij,kl} = D^{1x}_{ik}D^{1y}_{jl} + D^{2x}_{ik}D^{2y}_{jl}.
$$

Let us denote by $\mu_1$ the lowest eigenvalue of the tensor $A$, by $I$ the identity operator, and by $P_{\mu_1} \in \cL(\R^{N_x\times N_y})$ the orthogonal projector onto the eigenspace of 
$A$ associated with 
$\mu_1$. Figure~2 shows the decay of the error on the eigenvalues 
$\log_{10}(|\mu_1 - \lambda_n|)$ and of the error on the eigenvectors $\log_{10}(\|(I-P_{\mu_1})U_n\|_F)$, where $\|\cdot\|_F$ denotes the Frobenius norm of $\R^{N_x\times N_y}$,
 as a function of $n$ for the three algorithms and their orthogonal versions.  

These tests were performed with several matrices $D^{1x}, D^{1y}, D^{2x}, D^{2y}$, either drawn randomly or chosen such that the eigenspace associated 
with the lowest eigenvalue 
is of dimension greater than $1$. In any case, the three greedy algorithms converge towards a particular eigenstate associated with the lowest eigenvalue of the tensor $A$. Besides, the rate of convergence always seems to be exponential with respect to $n$. The error on the eigenvalues decays twice as fast as the error 
on the eigenvectors, as usual when dealing with the approximation of linear eigenvalue problems. 

We observe that the PRaGA and PEGA have similar convergence properties with respect to the number of iterations $n$. The behaviour of the PReGA
 strongly depends on the value $\nu$ chosen in (HA): the larger $\nu$, 
the slower the convergence of the PReGA. To ensure the efficiency of this method, it is important to choose the numerical parameter $\nu \in \R$ appearing in (\ref{eq:ascal})  as small as possible 
so that (HA) remains true. If the value of $\nu$ is well-chosen, the PReGA may converge as fast as the PRaGA or the PEGA, as illustrated in Section~\ref{sec:plate}.  
In the example presented in Figure~2 where $\nu$ is chosen to be $0$ and $\mu_1 \approx 116$, we can clearly see that the rate of convergence of the PReGA is poorer than the rates of the PRaGA and PEGA.

We also observe that the use of the ORaGA, OReGA and OEGA, instead of the pure versions of the algorithms, improves the convergence rate with respect to the number of iterations $n\in\N^*$. 
However, as $n$ increases, the cost of the $n$-dimensional optimization problems (\ref{eq:orthopt}) becomes more and more significant.

\begin{figure}[h]
\begin{center}
\psfrag{OrthoExp}[][][0.6]{OEGA}
\psfrag{Explicit}[][][0.6]{PEGA}
\psfrag{OrthoRes}[][][0.6]{OReGA}
\psfrag{Residual}[][][0.6]{PReGA}
\psfrag{OrthoRay}[][][0.6]{ORaGA}
\psfrag{Rayleigh}[][][0.6]{PRaGA}
\psfrag{logerrval}[][][0.8]{$\log_{10}\left(|\lambda_n - \mu_1|\right)$}
\psfrag{logerrvec}[][][0.8]{$\log_{10}\left(\|(I-P_{\mu_1})U_n\|_F\right)$}
\psfrag{n}[][][0.8]{$n$}
 \includegraphics*[width = 8cm]{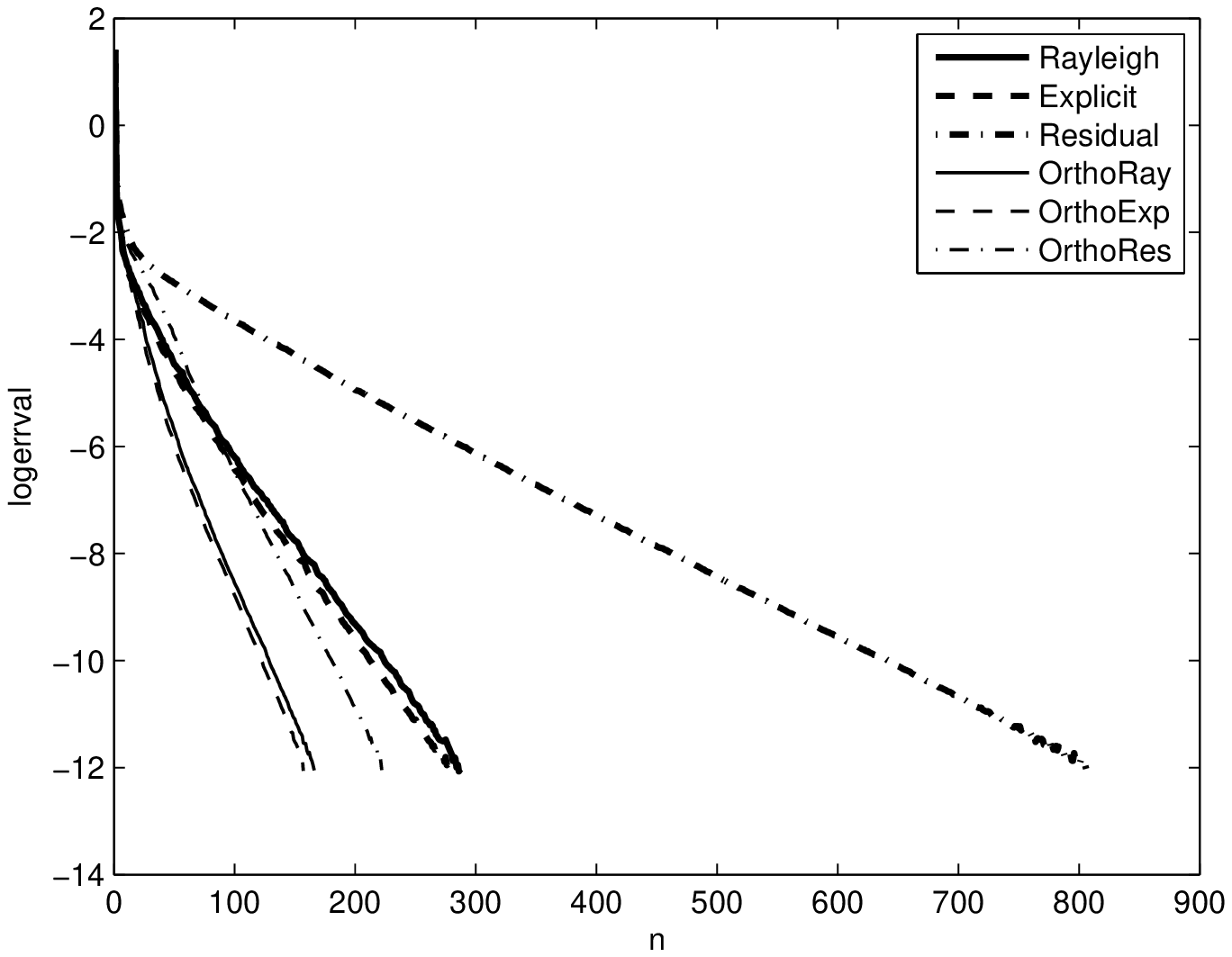}
 \includegraphics*[width = 8cm]{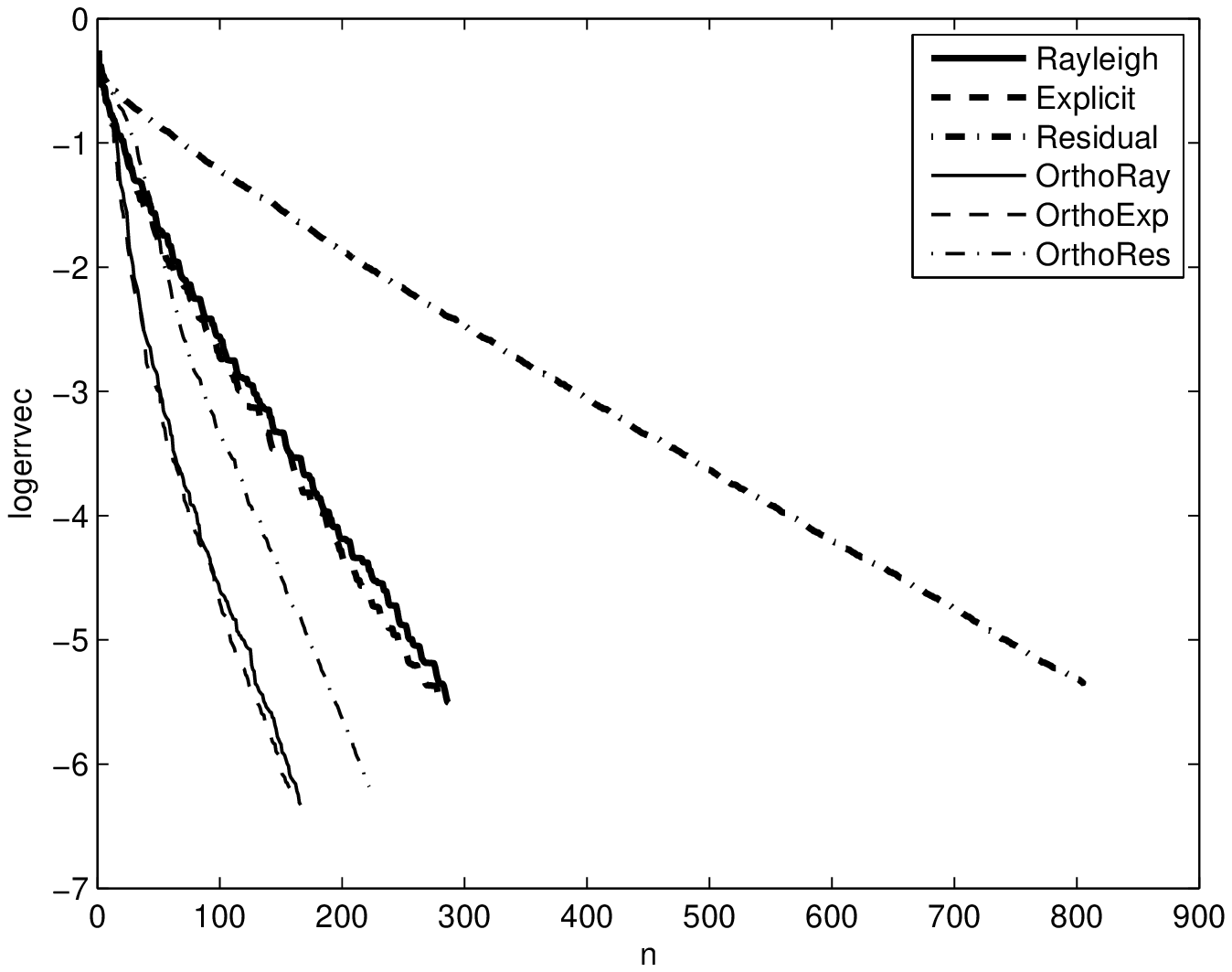}

\caption{Decay of the error of the three algorithms and their orthogonal versions: eigenvalues (left) and eigenvectors (right).}
\end{center}
\end{figure}

\subsection{First buckling mode of a microstructured plate with defects}\label{sec:plate}

We now consider the more difficult example of the computation of the first buckling mode of a plate~\cite{Bazant}. We first describe the continuous model, then detail its  discretization. 

\medskip

The plate is composed of two linear elastic materials, with different Young's moduli $E_1 = 1$ and $E_2 = 20$ respectively and same Poisson's ratio $\nu_{\rm P} = 0.3$. The rectangular reference configuration of the thin plate is $\Omega = \Omega_x\times \Omega_y$ with $\Omega_x = (0,1)$ and $\Omega_y = (0,2)$. 
The composition of the plate in the $(x,y)$ plane is represented in Figure~3: the black parts represent regions occupied by the first material and 
the white parts indicate the location of the second material.  

\begin{figure}[h]
\begin{center}
\psfrag{x}[][][1.0]{$x$ }
\psfrag{y}[][][1.0]{$y$ }
\psfrag{0}[][][0.7]{$0$}
\psfrag{1}[][][0.7]{$1$}
\psfrag{2}[][][0.7]{$2$}
\psfrag{Bas}[][][1.5]{$\Gamma_b$}
\psfrag{Cote}[][][1.5]{$\Gamma_s$}
\psfrag{Side}[][][1.5]{$\Gamma_s$}
\psfrag{Gamma}[][][1.5]{$\Gamma_t$}
 \includegraphics[width=6cm]{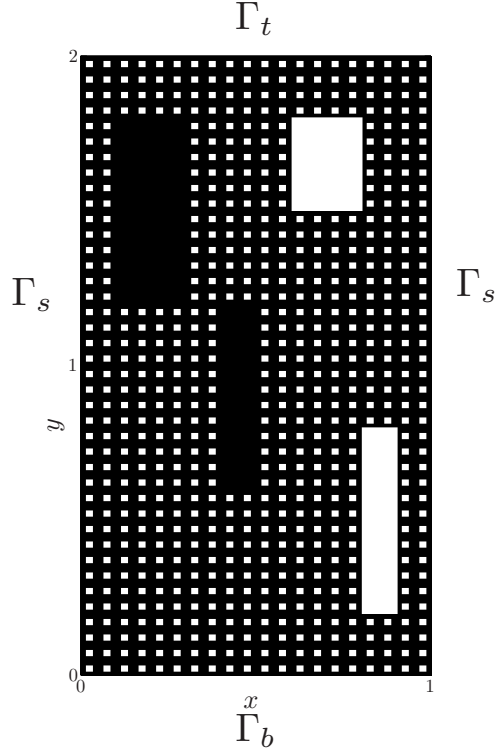}
\caption{Composition of the plate.}
\end{center}
\end{figure}

We denote by $(u_x,u_y, v): \Omega_x\times \Omega_y \to \R^3$ the displacement field of the plate, respectively in the $x$, $y$ and outer-plane direction, and by $u:=(u_x,u_y)$. 

The bottom part $\Gamma_b := [0,1]\times\{0\}$ of the plate is fixed, and a constant force $F = -0.05$ is applied in the $y$ direction on its top part
$\Gamma_t:= [0,1]\times\{2\}$. The sides of the plate $\Gamma_s:= \left(\{0\}\times \Omega_y\right) \cup \left(\{1\}\times \Omega_y \right)$ are free, 
and the outerplane displacement fields of the plate (and their derivatives) are imposed to be zero on the boundaries $\Gamma_b \cup \Gamma_t$.

The Hilbert spaces of kinematically admissible displacement fields are
$$
V^u:=\left\{u = (u_x,u_y)\in \left(H^1(\Omega_x \times \Omega_y)\right)^2, \; u_x (x,0) = u_y(x,0) = 0 \mbox{ for almost all } x\in \Omega_x \right\},
$$
and
$$
V^v := \left\{ v \in H^2(\Omega_x\times \Omega_y), \; v(x,0) = v(x,2)  = \frac{\partial v}{\partial y}(x,0)  = \frac{ \partial v}{\partial y}(x,2)  = 0 \mbox{ for almost all } x\in \Omega_x\right\}.
$$

For a displacement field $(u,v)\in V^u\times V^v$, the strain tensor is composed of two parts: 
\begin{itemize}
 \item a membrane strain: $\epsilon(u,v) = \epsilon_u(u) +\epsilon_v(v)$ is the sum of two parts; 
the first part $\epsilon_u(u)$ only depends on the inner-plane components of the displacement field
 $$
\epsilon_u (u):= \left[
                                          \begin{array}{cc}
                                          \frac{\partial u_x}{\partial x} & \frac{1}{2}\left( \frac{\partial u_x}{\partial y} + \frac{\partial u_y}{\partial x}\right)\\
                                          \frac{1}{2}\left( \frac{\partial u_x}{\partial y} + \frac{\partial u_x}{\partial y}\right) & \frac{\partial u_y}{\partial y}\\
                                          \end{array}
       \right],
$$
and the second part $\epsilon_v(v)$ only depends on the outer-plane component of the displacement field
$$
\epsilon_v(v) := \left[
                                          \begin{array}{cc}
                                           \frac{1}{2}\left( \frac{\partial v}{\partial x}\right)^2 & \frac{1}{2}\left( \frac{\partial v}{\partial x}\frac{\partial v}{\partial y}\right)\\
                                          \frac{1}{2}\left(\frac{\partial v}{\partial x}\frac{\partial v}{\partial y}\right) & \frac{1}{2}\left( \frac{\partial v}{\partial y}\right)^2\\
                                          \end{array}
    \right];
$$
\item a curvature strain: $\chi(v)$, which only depends on $v$:
 $$
\chi(v):= \left[
                                          \begin{array}{cc}
                                           \frac{\partial^2 v}{\partial x^2} &  \frac{\partial^2 v}{\partial x\partial y}\\
                                          \frac{\partial^2 v}{\partial x\partial y} & \frac{\partial^2 v}{\partial y^2}\\
                                          \end{array}
                                         \right].
 $$
\end{itemize}

The potential energy of the plate, 
\begin{align*}
W \; : \; V^u \times V^v & \to  \R \\
(u,v) & \mapsto  W(u,v)\\
\end{align*}
is defined as follows (we drop here the dependence in $(u,v)$ of the strain 
fields for the sake of clarity): for all $(u,v)\in V^u\times V^v$,
\begin{align*}
W(u,v)& :=  \int_{\Omega_x\times \Omega_y} \frac{E(x,y)h}{2(1-\nu_{\rm P}^2)} \left[ \nu_{\rm P} \left(\mbox{Tr} \epsilon\right)^2 + (1-\nu_{\rm P}) \epsilon :\epsilon \right]\,dx\,dy \quad \quad \mbox{(membrane energy)}\\
&+  \int_{\Omega_x\times \Omega_y} \frac{E(x,y)h^3}{24(1-\nu_{\rm P}^2)} \left[ \nu_{\rm P} \left(\mbox{Tr} \chi\right)^2 + (1-\nu_{\rm P})\chi : \chi \right]\,dx\,dy  \quad \quad \mbox{(bending energy)}\\
&-  \int_{\Omega_x} Fu_y(x, 2)\,dx, \quad \mbox{ (external forces)}\\
\end{align*}
where $h$ is the thickness of the plate.

A stationary equilibrium of the plate $(u^0, v^0)\in V^u\times V^v$ is a kinematically admissible displacement field such that $W'\left(u^0,v^0\right) = 0$, where $W'(u^0, v^0)$ denotes 
the derivative of $W$ at $(u^0,v^0)$. We consider here the particular stationary equilibrium of the plate $(u^0,v^0)\in V^u\times V^v$ such that $v^0 = 0$ and 
$u^0\in V^u$ is the unique solution of the minimization problem:
\begin{equation}\label{eq:minplate}
u^0 = \mathop{\mbox{\rm argmin}}_{u\in V^u} \cE(u),
\end{equation}
where $\cE:V^u \to \R$ is defined as
$$
\forall u \in V^u, \quad \cE(u):= \int_{\Omega_x\times \Omega_y} \frac{E(x,y)h}{2(1-\nu_{\rm P}^2)} \left[ \nu_{\rm P} \left(\mbox{Tr} \epsilon_u\right)^2 + (1-\nu_{\rm P})\epsilon_u : \epsilon_u \right]\,dx\,dy- \int_{\Omega_x} Fu_y(x, 2)\,dx.
$$
The energy functional $\cE$ depends quadratically on $V^u$, so that the minimization problem (\ref{eq:minplate}) can be solved numerically using standard (PGA or OGA) greedy algorithms for unconstrained 
minimization problems such as those described in Section~\ref{sec:convex}. 
A suitable dictionary can be chosen as:
$$
\Sigma^u:= \left\{ (r_x\otimes s_x, r_y\otimes s_y), \; r_x, r_y \in V_x^u, \; s_x, s_y \in V_y^u \right\},
$$
where
$$
V_x^u:=H^1(\Omega_x), \quad V_y^u:= \left\{ s\in H^1(\Omega_y), \; s(0) = 0 \right\}.
$$
Then, $\Sigma^u$ and $\cE$ satisfy assumptions (H$\Sigma 1$), (H$\Sigma 2$), (H$\Sigma 3$), (HE1) and (HE2), so that the theoretical convergence results mentioned in Section~\ref{sec:convex} 
hold. 

\medskip

There is \itshape buckling \normalfont 
if and only if the smallest eigenvalue of the Hessian $A^0:= W''(u^0,v^0)$ is negative. An associated eigenvector is the \itshape first buckling mode \normalfont 
of the plate. Since $v^0 = 0$, for all $(u^1, v^1), (u^2,v^2) \in V^u\times V^v$,  
$$
A^0\left( (u^1,v^1), (u^2,v^2) \right) = A^0_u(u^1,u^2) + A^0_v(v^1,v^2),
$$
where
 $$
 A^0_u(u^1,u^2):= 2 \int_{\Omega_x\times \Omega_y} \frac{E(x,y)h}{(1-\nu_{\rm P}^2)} \left[ \nu_{\rm P} \mbox{Tr}\epsilon_{u}(u^1) \mbox{Tr}\epsilon_{u}(u^2) + (1-\nu_{\rm P}) \epsilon_{u}(u^1) : \epsilon_{u}(u^2) \right]\,dx\,dy
 $$
and
 \begin{align*}
 A^0_v(v^1,v^2) & :=  2\int_{\Omega_x\times \Omega_y} \frac{E(x,y)h^3}{12(1-\nu_{\rm P}^2)} \left[ \nu_{\rm P} \mbox{Tr} \chi(v^1)\mbox{Tr}\chi(v^2) + (1-\nu_{\rm P})\chi(v^1) : \chi(v^2) \right]\,dx\,dy\\
 &+  2 \int_{\Omega_x\times \Omega_y} \frac{E(x,y)h}{(1-\nu_{\rm P}^2)} \left[ \nu_{\rm P} \mbox{Tr} \epsilon(u^0,v^0) \mbox{Tr} e(v^1,v^2) + (1-\nu_{\rm P})\epsilon(u^0,v^0) : e(v^1,v^2) \right]\,dx\,dy,\\
 \end{align*}
with
 $$
e(v^1,v^2) := \left[
 \begin{array}{cc}
  \frac{\partial v^1}{\partial x}\frac{\partial v^2}{\partial x} & \frac{1}{2}\left( \frac{\partial v^1}{\partial x}\frac{\partial v^2}{\partial y} + \frac{\partial v^1}{\partial y}\frac{\partial v^2}{\partial x} \right)\\
 \frac{1}{2}\left( \frac{\partial v^1}{\partial x}\frac{\partial v^2}{\partial y} + \frac{\partial v^1}{\partial y}\frac{\partial v^2}{\partial x} \right) &  \frac{\partial v^1}{\partial y}\frac{\partial v^2}{\partial y}\\
 \end{array}
 \right].
 $$

The smallest eigenvalue of $A^0$ is thus the minimum of the smallest eigenvalues of $A^0_u$ and $A^0_v$. The bilinear form $A^0_u$ is coercive on $V^u\times V^u$ 
so that its smallest eigenvalue is positive. Determining whether the plate buckles or not amounts to computing the smallest eigenvalue of $A^0_v$
 and checking its sign. We therefore have to compute the lowest eigenvalue of the linear elliptic eigenvalue problem
\begin{equation}\label{eq:eigplate}
\left\{
\begin{array}{l}
 \mbox{find }(v,\mu) \in V^v \times \R \mbox{ such that }\\
\forall w\in V^v, \; A^0_v(v,w) = \mu \langle v,w\rangle_{L^2(\Omega_x\times \Omega_y)}.
\end{array}
\right.
\end{equation}
The Hilbert spaces $V^v$ and $H^v:= L^2(\Omega_x \times \Omega_y)$ satisfy assumption (HV) and $A^0_v$ is a symmetric continuous bilinear form on $V^v \times V^v$ 
satisfying (HA). The set $\Sigma^v$ defined as
$$
\Sigma^v:= \left\{ r\otimes s, \; r\in V^v_x, \; s\in V^v_y \right\},
$$
where
$$
V^v_x:= H^2(\Omega_x) \quad \mbox{ and } V^v_y:= \left\{ s\in H^2(\Omega_y), \; s(0) = s'(0) = s(2) = s'(2) = 0\right\} ,
$$
forms a dictionary of $V^v$. The theoretical convergence results we presented for the PRaGA and the PReGA then hold. We are going to compare their numerical behaviour with the PEGA. 

\medskip

Let us now make precise the discretization spaces used in order to solve (\ref{eq:minplate}) and (\ref{eq:eigplate}). The spaces $V_x^u$ and $V_y^u$ are discretized using $\P_1$ finite elements over uniform meshes of size  $\Delta x=\Delta y=2 \times 10^{-3}$. Denoting by $\Vt_x^u$ and  $\Vt_y^u$ the so-obtained finite element spaces, and by $\left( \phi_i^x \right)_{0\leq i \leq N_x}$ and $ \left(\phi_j^y \right)_{0 \leq j \leq N_y}$ the corresponding canonical basis functions, the discretization space is equal to 
$$
\Vt^u:= \mbox{\rm Span}\left\{ (\phi_i^x \otimes \phi_j^y, \phi_{i'}^x \otimes \phi_{j'}^y), \; 0\leq i,i'\leq N_x, \; 1\leq j,j' \leq N_y \right\},
$$
and the discretized dictionary is given by
$$
\Sigmat^u:= \left\{ (\rt_x \otimes \st_x, \rt_y \otimes \st_y), \; \rt_x, \rt_y \in \Vt^u_x, \; \st_x, \st_y \in \Vt^u_y \right\}.
$$
The set $\Sigmat^u$ and the energy function $\widetilde{\cE}$, the restriction of $\cE$ to the space $\Vt^u$, still satisfy assumptions (H$\Sigma 1$), (H$\Sigma 2$), (H$\Sigma 3$), 
(HE1) and (HE2). 
The PGA can then be applied to approximate the solution $\widetilde{u}_0$ of the discretized version of (\ref{eq:minplate}): 
\begin{equation}\label{eq:minconvV}
\widetilde{u}_0  = \mathop{\rm argmin}_{\ut\in \Vt^u}\widetilde{\cE}(\ut).
\end{equation}
To solve (\ref{eq:minconvV}), we performed 30 iterations of the PGA (which was enough to ensure convergence) and obtained an approximation of $\widetilde{u}_0$, which is 
used in the computation of the buckling mode of the plate.

\medskip

Problem (\ref{eq:eigplate}) is discretized using Hermite finite elements (cubic splines) on the one-dimensional uniform meshes used to solve (\ref{eq:minplate}). This discretization method gives rise to the finite element discretization spaces $\Vt_x^v \subset V_x^v$ and $\Vt_y^v \subset V_y^v$. The total discretization space is then
$$
\Vt^v :=\Vt_x^v \otimes \Vt_y^v= \mbox{\rm Span} \Sigmat^v \quad \mbox{where} \quad 
\Sigmat^v:= \left\{ \rt\otimes \st , \; \rt \in \Vt^v_x, \; \st\in \Vt^v_y\right\}.
$$  
The discretized version of the eigenvalue problem (\ref{eq:eigplate}) consists in computing the lowest eigenstate of the problem 
\begin{equation}\label{eq:eigplatedisc}
\left\{
\begin{array}{l}
 \mbox{find }(v,\mu) \in \Vt^v \times \R \mbox{ such that }\\
\forall w\in \Vt^v, \; A^0_v(v,w) = \mu \langle v,w\rangle_{L^2(\Omega_x\times \Omega_y)}.
\end{array}
\right.
\end{equation} 
Assumptions (HV), (HA), (H$\Sigma 1$), (H$\Sigma 2$) and (H$\Sigma 3$) are also satisfied in this discretized setting so that the 
above greedy algorithms can be carried out. We have performed the PRaGA, PReGA and PEGA on this problem. The approximate eigenvalue is found to be $\lambda \approx 1.53$.
 At each iteration $n\in\N^*$, the algorithms produce an approximation $\lambda_n$ of the eigenvalue and an approximation $u_n\in \Vt^v$ of the associated eigenvector.

\medskip

The smallest eigenvalue $\mu_1$ of $A^0_v$, and an associated eigenvector $\psi_1$, are computed using an inverse power iteration algorithm. 
Figure~4 shows the decay of the error on the eigenvalue, and on the eigenvector in the $H^2(\Omega_x\times \Omega_y)$ norm 
as a function of $n\in\N^*$, for the PRaGA, PReGA and PEGA. More precisely, the quantities $ \log_{10}\left( \frac{|\lambda_n - \mu_1|}{|\mu_1|}\right)$
and $\log_{10}\left(\frac{\|u_n - \psi_1\|_{H^2(\Omega_x\times \Omega_y)}}{\|\psi_1\|_{H^2(\Omega_x\times \Omega_y)}}\right)$ are plotted as a function of $n\in \N^*$. 
As for the toy problem dealt with in the previous section, the numerical behaviors of the PEGA and PRaGA are similar. 
Besides, we observe that the rate of convergence of the PReGA is comparable with those of the other two algorithms. Let us note that we have chosen here $\nu = 0$. 

\begin{figure}[h]
\begin{center}
\psfrag{logerrl}[][][0.7]{$ \log_{10}\left( \frac{|\lambda_n - \mu_1|}{|\mu_1|}\right)$}
\psfrag{logerrvec}[][][0.7]{$\log_{10}\left(\frac{\|u_n - \psi_1\|_{H^2(\Omega_x\times \Omega_y)}}{\|\psi_1\|_{H^2(\Omega_x\times \Omega_y)}}\right)$}
\psfrag{n}[][][1.0]{$n$}
\psfrag{Residual}[][][0.7]{PReGA}
\psfrag{Rayleigh}[][][0.7]{PRaGA}
\psfrag{Explicit}[][][0.7]{PEGA}
 \includegraphics*[width = 8.2cm]{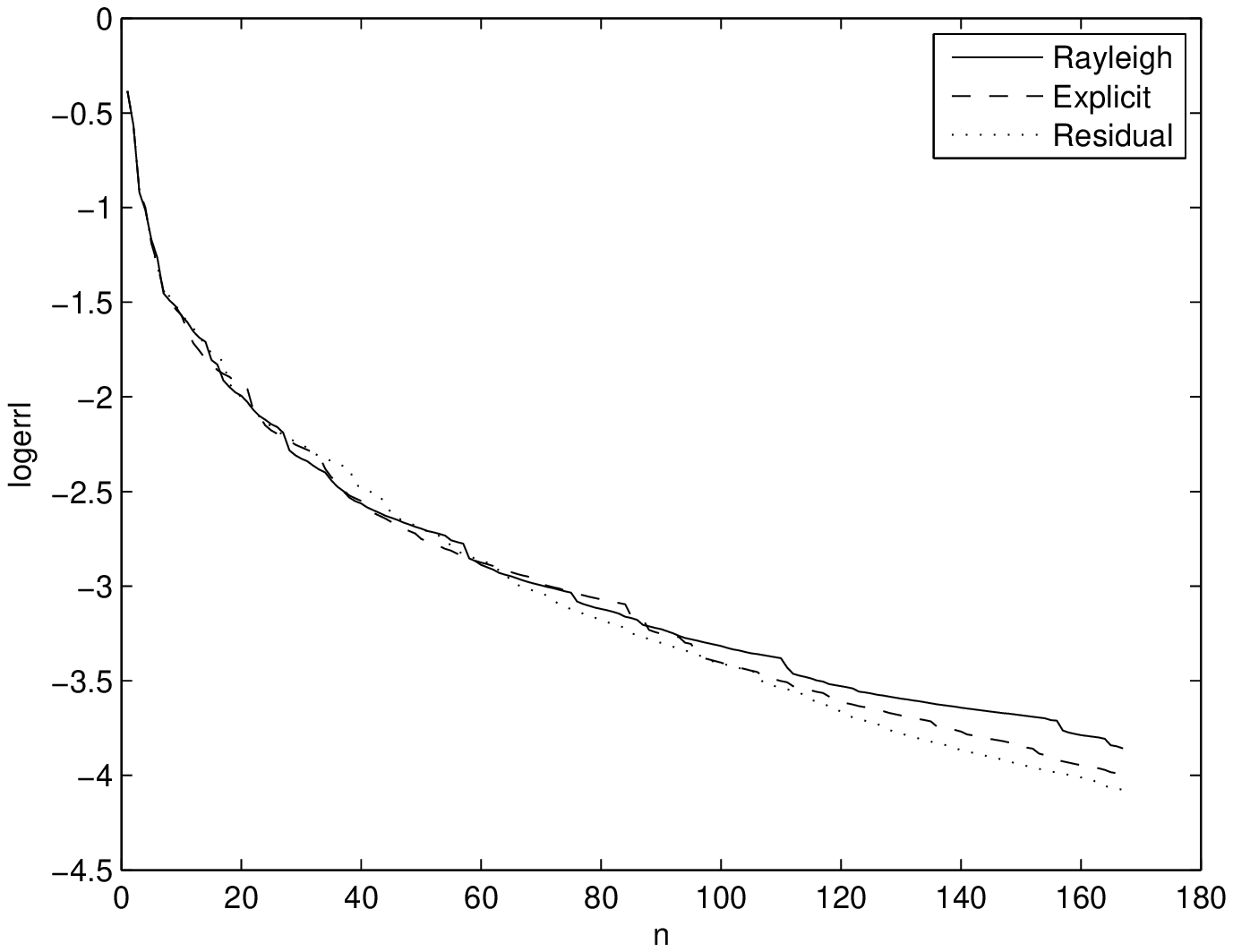}
 \includegraphics*[width = 8.2cm]{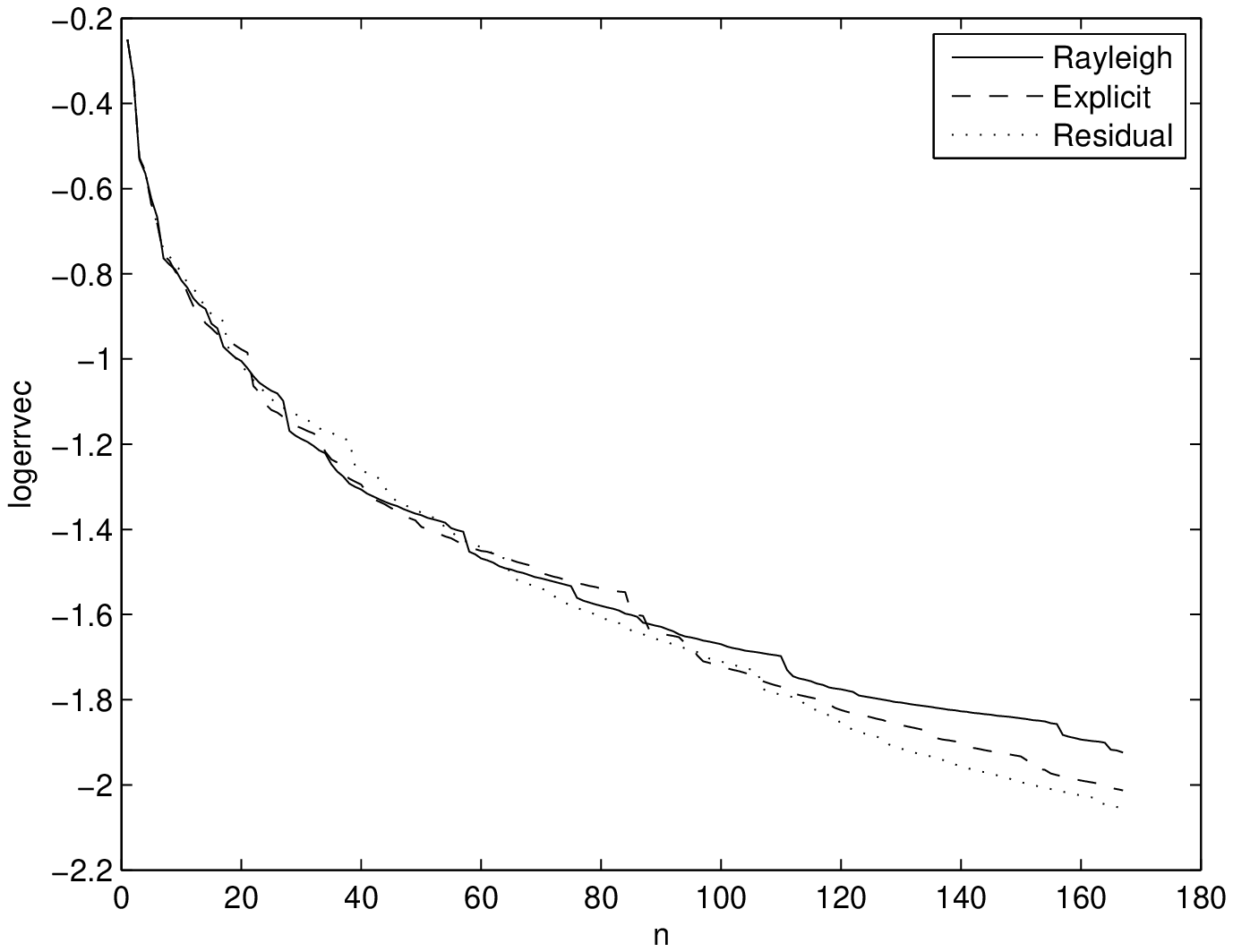}

\caption{Decay of the error as a function of $n$ for the PRaGA, PReGA and PEGA: on the eigenvalue (left) and on the eigenvector in the $H^2(\Omega_x\times \Omega_y)$ norm (right).}
\end{center}
\end{figure}

The isolines of the approximation $u_n$ given by the PRaGA are drawn in Figure~5 for different values of $n$ 
(the approximations given by the other two algorithms 
are similar). We can observe that the influence of the different defects of the plate appears gradually with $n\in\N^*$. 

\begin{figure}[h]
\begin{center}
\psfrag{x}[][][1.0]{$x$}
\psfrag{y}[][][1.0]{$y$}
\psfrag{n=1}[][][1.0]{$n=1$}
\psfrag{n=5}[][][1.0]{$n=5$}
\psfrag{n=50}[][][1.0]{$n=50$}
 \includegraphics*[width=5.5cm]{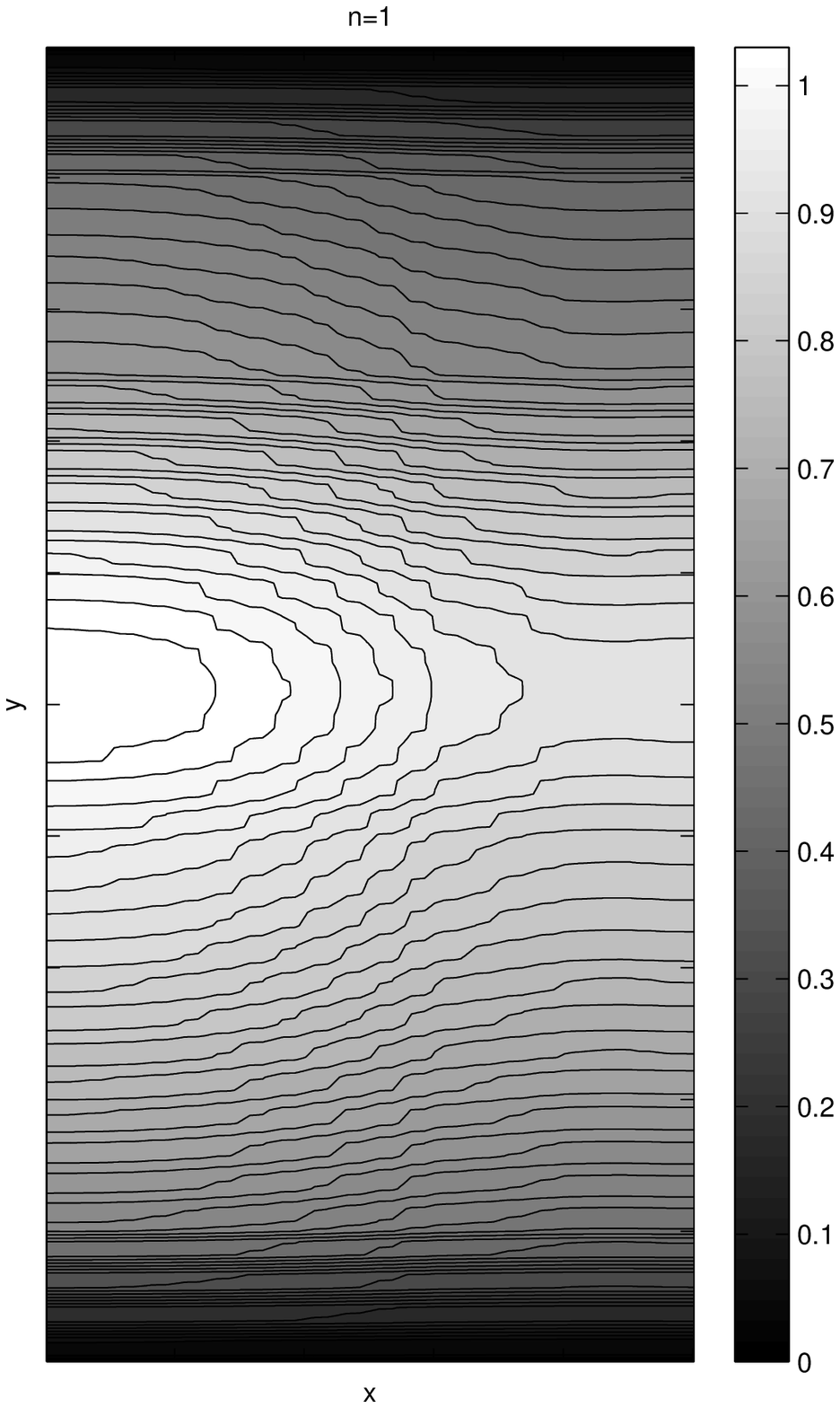}
 \includegraphics*[width=5.5cm]{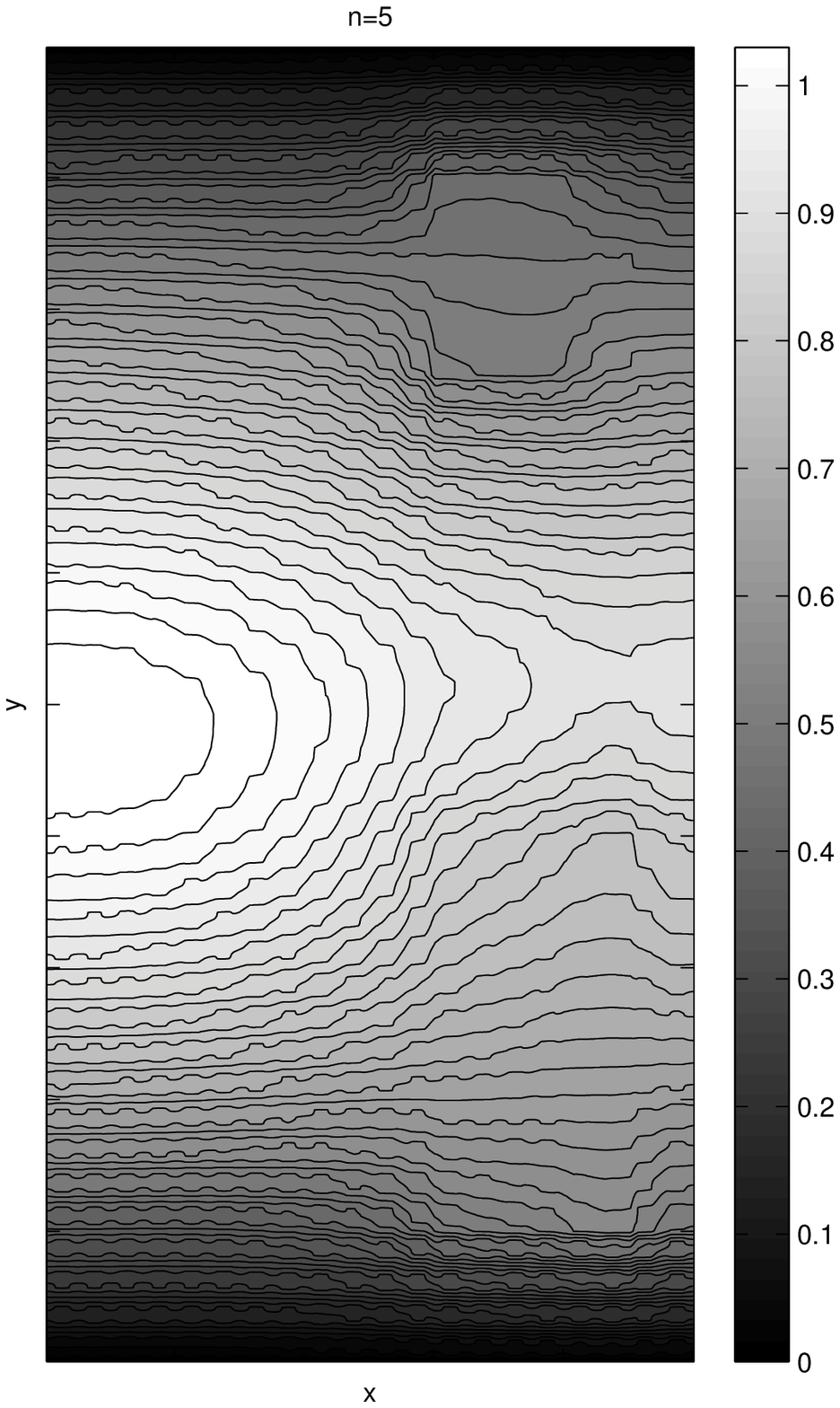}
 \includegraphics*[width=5.5cm]{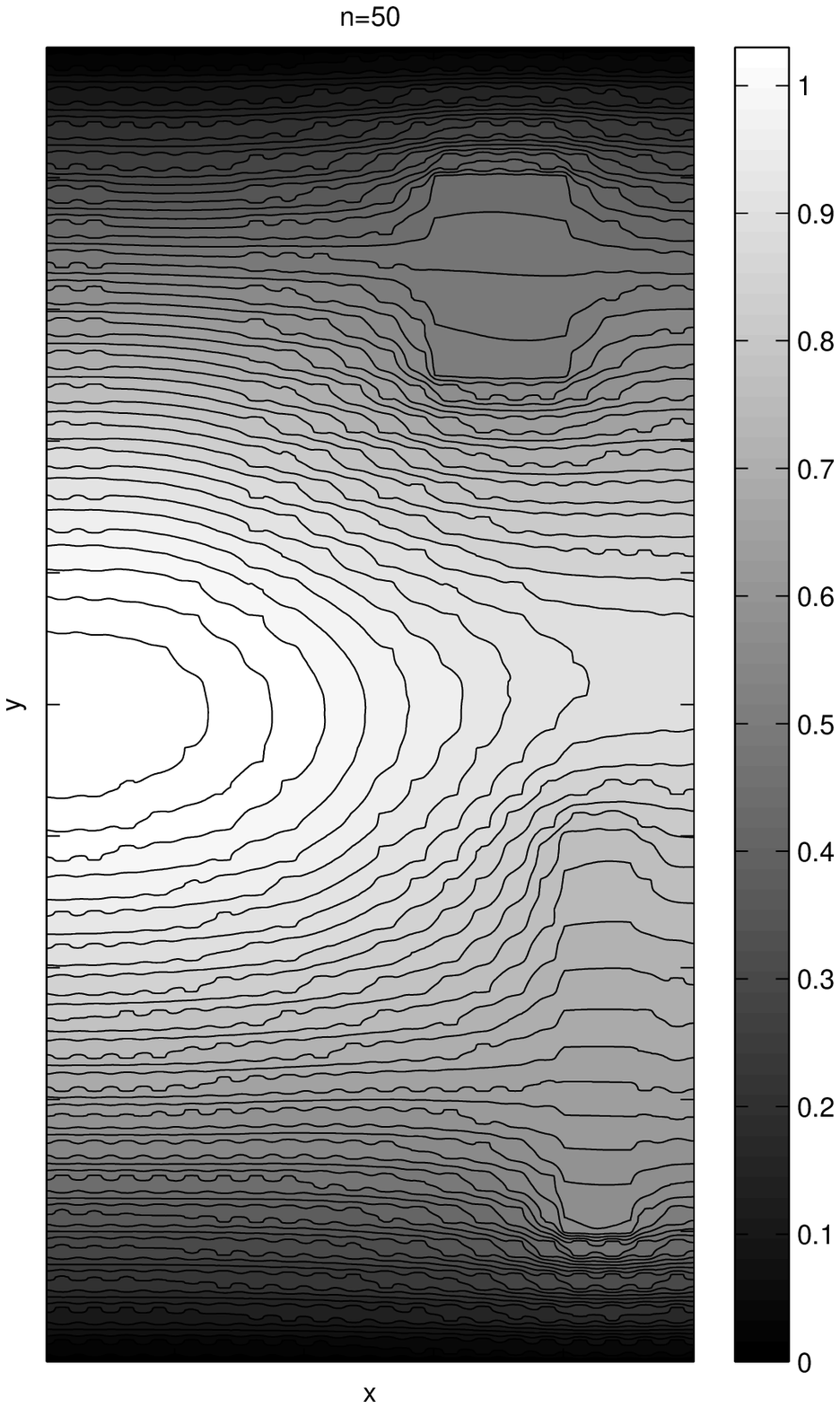}
\caption{Isolines of the approximation of the first buckling mode of the plate given by the Rayleigh quotient algorithm for $n=1$ (left), $n=5$ (center) and $n=50$ (right).}
\end{center}
\end{figure}

\section{Proofs}\label{sec:proof}

\subsection{Proof of  Lemma~\ref{lem:leminit}}

\noindent
{\em Proof that $(i) \Rightarrow (ii)$}

\medskip

Let $z\in \Sigma$. For $\varepsilon >0$ small enough so that $\varepsilon \|z\| < \|w\|$, $w + \varepsilon z \neq 0$. Then, since $\|w\|=1$, $(i)$ implies that
$$
\cJ(w+\varepsilon z)-\cJ(w) = \frac{\left( 2\varepsilon a(w,z) + \varepsilon^2 a(z, z)\right) - \left( 2\varepsilon \langle w, z\rangle + \varepsilon^2 \|z\|^2\right) a(w,w)}{ \|w+ \varepsilon z\|^2} \geq 0.
$$
Letting $\varepsilon$ go to zero, this yields
$$
 a(w,z) - a(w,w) \langle w, z\rangle = 0 \quad \mbox{and} \quad 
  a(z, z) - \|z\|^2 a(w,w) \geq 0.
$$
Using assumption (H$\Sigma 3$), we obtain
$$
\lambda_w := a(w,w) \leq \mathop{\inf}_{z\in \Sigma^*} \frac{a(z, z)}{\|z\|^2}  = \lambda_\Sigma \quad \mbox{and} \quad 
\forall v \in V, \; a(w,v) = \lambda_w \langle w , v \rangle,
$$
where $\Sigma^*$ is defined by (\ref{eq:defsigstar}). Hence \itshape(ii)\normalfont.

\medskip

\noindent
{\em Proof that \itshape(ii)\normalfont $\Rightarrow$ \itshape(i)\normalfont}

\medskip

Using \itshape(ii)\normalfont, similar calculations yield that for all $z\in \Sigma$ such that $w+z \neq 0$, 
\begin{align*}
\cJ(w+z) - \cJ(w)  & = \frac{a(w+ z, w + z)}{\|w +  z\|^2} - a(w,w)\\
& = \frac{ 2 a(w,z) +  a(z,z) - \left( 2 \langle w, z\rangle + \|z\|^2\right) a(w,w)}{ \|w+ z\|^2}\\
& = \frac{a(z,z) - \lambda_w \|z\|^2}{\|w+z\|^2}.\\
\end{align*}
This implies that $\cJ(w+z) - \cJ(w) \geq 0$. Hence \itshape(i) \normalfont since the inequality is trivial in the case when $w + z = 0$.

\subsection{Proof of  Lemma~\ref{lem:lemma2}}

 Let us first prove that (\ref{eq:minipb}) has at least one solution in the case when $w=0$. Let $(z_m)_{m\in\N^*}$ be a minimizing sequence:  $\forall m\in \N^*$, $z_m\in \Sigma$, $\|z_m\| = 1$ and $\dps a(z_m , z_m) \mathop{\longrightarrow}_{m\to \infty} \lambda_\Sigma$. 
The sequence $\left( \|z_m\|_a \right)_{m\in \N^*}$ being bounded, there exists $z_* \in V$ such that $(z_m)_{m\in\N^*}$ 
weakly converges, up to extraction, to some $z_*$ in $V$. By (H$\Sigma 2$), $z_*$ belongs to $\Sigma$. Besides, using (HV), the sequence $(z_m)_{m\in\N^*}$ 
strongly converges to $z_*$ in $H$, so that $\|z_*\| = 1$. Lastly, 
$$
\|z_*\|_a \leq \mathop{\lim}_{m \to \infty} \|z_m\|_a,
$$
which implies that $\dps a(z_*, z_*) = \cJ(z_*) \leq \lambda_\Sigma = \mathop{\lim}_{m\to\infty} a(z_m, z_m)$. Hence, $z_*$ is a minimizer of problem 
(\ref{eq:minipb}) when $w=0$. 

\medskip

Let us now consider $w\in V\setminus \Sigma$ and $(z_m)_{m\in \N^*}$ a minimizing sequence for problem (\ref{eq:minipb}). There exists $m_0\in \N^*$ large enough such that 
for all $m\geq m_0$, $w+z_m \neq 0$.
Let us denote by $\alpha_m:= \frac{1}{\|w + z_m\|}$ and $\zt_m:= \alpha_m z_m$. It holds that $\|\alpha_m w + \zt_m \| = 1$ and 
$\dps a(\alpha_m w + \zt_m, \alpha_m w + \zt_m) \mathop{\longrightarrow}_{m\to \infty} \mathop{\inf}_{z\in \Sigma} \cJ(w + z)$. 

If the sequence $(\alpha_m)_{m\in\N^*}$ is bounded, then so is the sequence $(\|\zt_m\|_a)_{m\in\N^*}$, and reasoning as above, we can prove that there exists a minimizer to problem 
(\ref{eq:minipb}).

To complete the proof, let us now argue by contradiction and assume that, up to the extraction of a subsequence, 
$\dps \alpha_m \mathop{\longrightarrow}_{m\to\infty} + \infty$. 
Since the sequence $\left(\| \alpha_m w + \zt_m\|_a \right)_{m\in\N^*}$ is bounded and for all $m\in\N^*$,
$$
\| \alpha_m w + \zt_m\|_a = \alpha_m \| w + z_m\|_a, 
$$
the sequence $(z_m)_{m\in\N^*}$ strongly converges towards $-w$ in $V$. Using assumption (H$\Sigma 2$), this implies that $w \in \Sigma$, which leads to a contradiction.

\subsection{Proof of Theorem~\ref{th:main} for the PRaGA}\label{sec:proofRayleigh}

Throughout this section, we use the notation of Section~\ref{sec:rayleigh}. Let us point out that from Lemma~\ref{lem:Radef}, all the iterations of the PRaGA are well-defined and the sequence 
$(\lambda_n)_{n\in\N}$ is non-increasing. 

\begin{lemma}\label{lem:ELrayleigh}
For all $n \geq 1$, it holds
\begin{equation}\label{eq:ELn}
a(u_n, z_n) - \lambda_n \langle u_n, z_n\rangle = 0. 
\end{equation}
\end{lemma}

\begin{proof}
Let us define  $\mathcal{S}: \R \ni t\mapsto \cJ(u_{n-1} + tz_n)$. From Lemma~\ref{lem:Radef}, since $\lambda_0 < \lambda_\Sigma$, all the iterations of 
the PRaGA are well-defined and for all $n\in\N^*$, we have $u_{n-1}\notin \Sigma$. Hence, since $\Sigma$ satisfies (H$\Sigma 1$), for all $t\in \R$, $tz_n\in \Sigma$ and 
$u_{n-1} + tz_n \neq 0$. The function $\mathcal{S}$ is thus differentiable on $\R$ and admits a minimum at $t=1$. The first-order Euler equation at $t=1$ reads
$$
\frac{1}{\|u_{n-1} + z_n\|^2}\left( a(u_{n-1} + z_n, z_n) - \lambda_n \langle u_{n-1}+z_n, z_n\rangle \right) = 0,
$$
which immediatly leads to (\ref{eq:ELn}). 
\end{proof}

In the rest of this Section, we will denote by $\alpha_n = \frac{1}{\|u_{n-1} + z_n\|}$ and $\zt_n = \frac{z_n}{\|u_{n-1} + z_n\|}$, 
so that for all $n\in\N^*$, $u_n = \alpha_n u_{n-1} + \zt_n$. We first prove the following intermediate lemma.

\begin{lemma}\label{lem:lemCV1}
The series $\sum_{n=1}^{+\infty} \|\zt_n\|^2$ and $\sum_{n=1}^{+\infty} \|\zt_n\|_a^2$ are convergent.
\end{lemma}
\begin{proof}
Let us first prove that the series $\sum_{n=1}^{+\infty} \|\zt_n\|^2$ is convergent. For all $n\in\N^*$, we have
 $$
 a(u_n, u_n) = \frac{a(u_{n-1} + z_n, u_{n-1} + z_n)}{\|u_{n-1} + z_n\|^2}.
 $$
 Thus, using (\ref{eq:ELn}) at the fifth equality,
\begin{align}
\lambda_{n-1} - \lambda_n &= a(u_{n-1}, u_{n-1}) - a(u_n, u_n) \nonumber \\
&=   \frac{a(u_{n-1}, u_{n-1})\left( 2 \langle u_{n-1}, z_n\rangle + \|z_n\|^2\right) -2a(u_{n-1}, z_n) - a(z_n, z_n)}{\|u_{n-1} + z_n\|^2} \nonumber\\ 
&=   \frac{2\left( \lambda_{n-1} \langle u_{n-1} + z_n, z_n \rangle -a(u_{n-1} + z_n, z_n)\right)-\lambda_{n-1}\|z_n\|^2 + a(z_n, z_n) }{\|u_{n-1} + z_n\|^2} \nonumber \\
&=   2\left( \lambda_{n-1}\langle u_n, \zt_n\rangle -a(u_n, \zt_n)\right) + a(\zt_n, \zt_n) - \lambda_{n-1} \|\zt_n\|^2 \nonumber\\
&=  2(\lambda_{n-1} - \lambda_n)\langle u_n, \zt_n\rangle  + a(\zt_n , \zt_n) - \lambda_{n-1} \|\zt_n\|^2 \label{eq:5eq}\\
&\geq   (\lambda_{\Sigma} - \lambda_{n-1}) \|\zt_n\|^2 - 2(\lambda_{n-1} - \lambda_n)|\langle u_n, \zt_n\rangle| \nonumber\\
&\geq   (\lambda_{\Sigma} - \lambda_{n-1}) \|\zt_n\|^2 - 2(\lambda_{n-1} - \lambda_n)\|u_n\|\|\zt_n\| \nonumber \\
&\geq    (\lambda_{\Sigma} - \lambda_{n-1}) \|\zt_n \|^2 - (\lambda_{n-1} - \lambda_n)\|\zt_n\|^2 - (\lambda_{n-1} - \lambda_n). \nonumber \\ \nonumber
\end{align}
 This implies that
 \begin{equation}\label{eq:est}
 2(\lambda_{n-1} - \lambda_n) \geq \left[ (\lambda_\Sigma - \lambda_{n-1}) - (\lambda_{n-1} - \lambda_n) \right] \|\zt_n\|^2.
 \end{equation}
From Lemma~\ref{lem:Radef}, $(\lambda_n)_{n\in\N}$ is a non-increasing sequence. Besides, since it is bounded from below by $\mu_1 = {\mathop{\min}}_{v\in V} \cJ(v)$, it converges 
towards a real number $\dps \lambda = \mathop{\lim}_{n\to +\infty} \lambda_n$ which satisfies $\lambda \leq \lambda_0 < \lambda_\Sigma$. Estimate (\ref{eq:est}) 
implies that there exists $\delta >0$ and $n_0\in \N^*$ such that for all $n\geq n_0$, 
$$
 2(\lambda_{n-1} - \lambda_n) \geq \delta \|\zt_n\|^2.
$$
Hence, the series $\sum_{n=1}^{+\infty} \|\zt_n\|^2$ is convergent, since the series $\sum_{n=1}^{+\infty}( \lambda_{n-1} - \lambda_n)$ is obviously convergent. 

\medskip

Let us now prove that the series $\sum_{n=1}^{+\infty} \|\zt_n\|_a^2$ is convergent. Using (\ref{eq:5eq}), it holds
\begin{align*}
\lambda_{n-1} - \lambda_n &= 2(\lambda_{n-1} - \lambda_n)\langle u_n, \zt_n\rangle  + a(\zt_n, \zt_n) - \lambda_{n-1} \|\zt_n\|^2\\
&\geq  -2(\lambda_{n-1} - \lambda_n)\|u_n\|\|\zt_n\|  + a(\zt_n, \zt_n) - \lambda_{n-1} \|\zt_n\|^2\\
&\geq   - (\lambda_{n-1} - \lambda_n) \|\zt_n\|^2 - (\lambda_{n-1} - \lambda_n) + a(\zt_n, \zt_n) - \lambda_{n-1} \|\zt_n\|^2.
\end{align*}
Thus, 
$$
2 (\lambda_{n-1} - \lambda_n) + (\nu +\lambda_{n-1} + (\lambda_{n-1} - \lambda_n)) \|\zt_n\|^2 \geq \|\zt_n\|_a^2,
$$
which implies that the series $\sum_{n=1}^{+\infty} \|\zt_n\|_a^2$ is convergent since $\nu + \lambda \geq \nu + \mu_1 >0$. 
 \end{proof}

 \begin{proof}[Proof of Theorem~\ref{th:main}]
We know that the sequence $(\lambda_n)_{n \in \N}$ converges to $\lambda$ which implies that the sequence $(\|u_n\|_a)_{n\in\N}$ is bounded. 
Thus, the sequence $(u_n)_{n\in\N}$ converges, up to the extraction of a subsequence, to some $w\in V$, weakly in $V$, and strongly in $H$ from (HV). Let us denote by 
$(u_{n_k})_{k\in\N}$ such a subsequence.
In particular, $\dps \|w\| = \mathop{\lim}_{k\to +\infty} \|u_{n_k}\| = 1$. Let us prove that $w$ is an eigenvector of the bilinear form $a(\cdot, \cdot)$ 
associated to $\lambda$ and that the sequence $(u_{n_k})_{k\in\N}$ strongly converges in $V$ toward $w$. 

Lemma~\ref{lem:lemCV1} implies that $\dps \zt_n \mathop{\longrightarrow}_{n\to\infty} 0$ strongly in $V$, and since $\|u_n\| = \|\alpha_n u_{n-1} + \zt_n\|  = \|u_{n-1} \| = 1$ 
for all $n\in \N^*$, necessarily
 $\dps \alpha_n\mathop{\longrightarrow}_{n\to\infty} 1$. Thus, $z_n = \frac{1}{\alpha_n}\zt_n$ also converges to $0$ strongly in $V$.

Besides, for all $n\geq 1$ and all $z\in \Sigma$, it holds that
$$
\cJ(u_{n-1} + z) \geq \cJ(u_{n-1} + z_n). 
$$
Using the fact that $\|u_{n-1}\| = 1$ and $a(u_{n-1}, u_{n-1})=\lambda_{n-1} $, this inequality also reads 
\begin{equation}\label{eq:eqRanew}
 \begin{array}{l}
\lambda_{n-1} \left[ 2 \langle u_{n-1}, z_n\rangle + \|z_n\|^2 -2 \langle u_{n-1} , z\rangle - \|z\|^2\right] \\
 \quad + \left[ 2a(z, u_{n-1}) + a(z,z) \right] \left[ 1 + 2\langle u_{n-1}, z_n\rangle + \|z_n\|^2\right] \\
 \quad 
- \left[2a(u_{n-1}, z_n) + a(z_n, z_n)\right]\left[1 + 2\langle u_{n-1} , z \rangle + \|z\|^2\right] \geq 0.  \\
 \end{array}
\end{equation}
Besides, $(z_n)_{n\in\N^*}$ strongly converges to $0$ in $V$ and $(\lambda_n)_{n\in\N}$ converges towards $\lambda$. As a consequence, taking $n = n_k+1$ in (\ref{eq:eqRanew}) 
and letting $k$ go to infinity, it holds that 
for all $z\in \Sigma$, 
 $$
- 2\lambda \langle w, z\rangle - \lambda \|z\|^2 + 2a(w,z)+ a(z,z) \geq 0.
 $$
Besides, from (H$\Sigma 1$), for all $\varepsilon >0$ and $z\in \Sigma$, $\varepsilon z\in \Sigma$. Thus, taking $\varepsilon z$ instead of $z$ in the above inequality yields
\begin{equation}\label{eq:epseq}
- 2\lambda \varepsilon \langle w, z\rangle - \lambda\varepsilon^2 \|z\|^2 + 2\varepsilon a(w,z)+ \varepsilon^2 a(z,z) \geq 0.
\end{equation}
Letting $\varepsilon$ go to $0$ in (\ref{eq:epseq}), we obtain that for all $z\in \Sigma$, 
$$
a(w,z) = \lambda \langle w, z\rangle \quad \mbox{and} \quad a(z,z) \geq \lambda \|z\|^2.
$$
Thus, using (H$\Sigma 3$), this implies that for all $v\in V$, $a(w,v) = \lambda \langle w, v\rangle$ 
and $w$ is an $H$-normalized eigenvector of $a(\cdot, \cdot)$ associated to the eigenvalue $\lambda$. Besides, since $\dps a(w,w) = \mathop{\lim}_{k\to\infty} a(u_{n_k}, u_{n_k})$ and 
$\dps \|w\|  = \mathop{\lim}_{k\to\infty} \|u_{n_k}\|$, it holds that $\dps \|w\|_a = \mathop{\lim}_{k\to\infty} \|u_{n_k}\|_a$ and 
the convergence of the subsequence $(u_{n_k})_{k\in\N}$ towards $w$ also holds strongly in  
$V$.

\medskip

Let us prove now that $\dps d_a(u_n, F_\lambda) \mathop{\longrightarrow}_{n\to\infty} 0$. Let us argue by contradiction and assume that there exists $\varepsilon>0$ and
 a subsequence $(u_{n_k})_{k\in\N}$ such that $d_a(u_{n_k}, F_\lambda) \geq \varepsilon$. Up to the extraction of another subsequence, from the results proved above, there exists 
$w\in F_\lambda$ such that $u_{n_k} \rightarrow w$ strongly in $V$. Thus, along this subsequence, 
$$
d_a(u_{n_k}, F_\lambda) \leq \|u_{n_k} - w\|_a \mathop{\longrightarrow}_{n\to\infty} 0,
$$
yielding a contradiction. 

\medskip

Lastly, if $\lambda$ is a simple eigenvalue, the only possible limits of subsequences of $(u_n)_{n\in\N}$ are $w_\lambda$ and $-w_\lambda$ where $w_\lambda$ is an $H$-normalized 
eigenvector associated with $\lambda$. As $(z_n)_{n\in\N^*}$ strongly converges to $0$ in $V$, 
the whole sequence $(u_n)_{n\in\N}$ converges, either to $w_\lambda$ or to $-w_\lambda$, and the convergence holds strongly in $V$.  
\end{proof}

\subsection{Proof of Theorem~\ref{th:main} for the Residual algorithm}\label{sec:proofResidual}

Throughout this section, we use the notation of Section~\ref{sec:residual}. From Lemma~\ref{lem:Redef}, we know that all the iterations of the PReGA are well-defined. 
The following lemma is the analog of Lemma~\ref{lem:ELrayleigh}.

\begin{lemma}\label{lem:ELresidual}
For all $n\geq1$, it holds
\begin{equation}\label{eq:Euler2}
\langle u_{n-1} +z_n, z_n \rangle_a - (\lambda_{n-1}+\nu)\langle u_{n-1}, z_n\rangle = 0. 
\end{equation}
\end{lemma}

As above, we set $\alpha_n:= \frac{1}{\|u_{n-1} + z_n\|}$ and $\zt_n := \alpha_n z_n$ so that 
$u_n = \alpha_n u_{n-1} + \zt_n$.

\begin{lemma}\label{lem:lemma3}
 The sequence $(\lambda_n)_{n\in\N}$ is non-increasing and the series $\sum_{n=1}^{+\infty} \|\zt_n\|^2 $ and $\sum_{n=1}^{+\infty} \|\zt_n\|_a^2$ are convergent.
\end{lemma}

\begin{proof}
Let us first prove that for all $n\in \N^*$,
$$
\lambda_n - \lambda_{n-1} =  \frac{a(u_{n-1}+z_n, u_{n-1} + z_n )}{\|u_{n-1} + z_n\|^2} - \lambda_{n-1} \leq 0.
$$
Since $\lambda_{n-1} = a(u_{n-1}, u_{n-1})$, using equation (\ref{eq:Euler2}), we obtain
$$
a(u_{n-1}+z_n, u_{n-1} + z_n ) - \lambda_{n-1}\|u_{n-1} + z_n\|^2 = - \|z_n\|_a^2 - (\lambda_{n-1} + \nu)\|z_n\|^2 \leq 0.
$$
Thus, the sequence $(\lambda_n)_{n\in\N}$ is non-increasing. Besides, since
\begin{equation}\label{eq:estnew}
\lambda_n - \lambda_{n-1} \leq - \|\zt_n\|_a^2 - (\lambda_{n-1} + \nu)\|\zt_n\|^2,
\end{equation}
this also yields that the series $\sum_{n=1}^{+\infty} \|\zt_n\|^2$ and $\sum_{n=1}^{+\infty} \|\zt_n\|_a^2$ are convergent, since 
$\dps \nu + \lambda = \nu + \mathop{\lim}_{n\to +\infty} \lambda_n  >0$.
\end{proof}

\begin{proof} [Proof of Theorem~\ref{th:main}]
 As for the PRaGA, Lemma~\ref{lem:lemma3} implies that $\dps \alpha_n \mathop{\longrightarrow}_{n\to\infty} 1$, $(z_n)_{n\in\N^*}$ converges 
to $0$ strongly in $V$ and $\dps \langle u_{n-1} , u_n \rangle \mathop{\longrightarrow}_{n\to\infty} 1$. Besides, the sequence $(u_n)_{n\in\N}$ is bounded in $V$,
 so, up to the extraction of a subsequence $(u_{n_k})_{k\in\N}$, there exists $w\in V$ such that $(u_{n_k})_{k\in\N}$ weakly converges towards $w$ in $V$. 

\medskip

From (\ref{eq:algo3}), we know that for all $n\in\N^*$ and all $z\in \Sigma$, 
$$
\frac{1}{2}\|u_{n-1} + z\|_a^2 - (\lambda_{n-1} + \nu) \langle u_{n-1}, z\rangle\geq  \frac{1}{2}\|u_{n-1} + z_n\|_a^2 - (\lambda_{n-1} + \nu) \langle u_{n-1}, z_n\rangle ,
$$
which leads to
\begin{equation}\label{eq:eqRenew}
 \langle u_{n-1}, z\rangle_a + \frac{1}{2}\|z\|_a^2 - (\lambda_{n-1} + \nu) \langle u_{n-1},z\rangle \geq \langle u_{n-1}, z_n\rangle_a + \frac{1}{2}\|z_n\|_a^2 - (\lambda_{n-1} + \nu) \langle u_{n-1}, z_n\rangle.
\end{equation}
Taking $n = n_k+1$ in (\ref{eq:eqRenew}) and letting $k$ go to infinity, we obtain
$$
\langle w, z\rangle_a  + \frac{1}{2}\|z\|_a^2 - (\lambda + \nu) \langle w,z\rangle \geq 0.
$$
which implies, by taking $\varepsilon z$ instead of $z$ in the equation above and letting $\varepsilon$ go to zero (which we can do because of (H$\Sigma$1)), 
$$
\langle w, z\rangle_a - (\lambda + \nu) \langle w, z\rangle = 0. 
$$
We infer from assumption (HV) and (H$\Sigma 3$) that $w$ is an $H$-normalized eigenvector of $a$ associated to the eigenvalue $\lambda$. 

\medskip

The rest of the proof uses exactly the same arguments as those used in the previous section. 
\end{proof}

\subsection{Proof of Proposition~\ref{prop:noHV} for the PReGA}\label{sec:proofnoHV}

In this section, we do not assume any more that the embedding $V \hookrightarrow H$ is compact, 
but we make the additional assumption that $\lambda_0 < \min \sigma_{\rm ess}(A)$.  From Lemma~\ref{lem:Redef}, we know that all the iterations of the PReGA are well-defined. 

\begin{proof}[Proof of Proposition~\ref{prop:noHV}]
Reasoning as in Section~\ref{sec:proofResidual}, we can prove that the sequence $(\lambda_n)_{n\in\N}$ is non-increasing and thus converges towards a limit $\lambda$. Besides, the series 
$\sum_{n=1}^{+\infty} \|\zt_n\|^2 $ and $\sum_{n=1}^{+\infty} \|\zt_n\|_a^2$ are convergent. We also know that the sequence $(u_n)_{n\in\N}$ is bounded in $V$. We can therefore extract
 from $(u_n)_{n\in\N}$ a subsequence $(u_{n_k})_{k\in\N}$ which weakly converges in $V$ towards some $w\in V$ satisfying 
$\dps \|w\| \leq \mathop{\mbox{liminf}}_{k\to +\infty} \|u_{n_k}\| = 1$.
 Besides, still reasoning as in Section~\ref{sec:proofResidual}, we obtain
\begin{equation}\label{eq:eqw}
\langle w, z\rangle_a - (\lambda + \nu) \langle w, z\rangle = 0, 
\end{equation}
for all $z\in \Sigma$. 

Let us first prove that in fact $\|w\| = 1$. Without loss of generality, up to adding a constant to the operator $A$, we can assume that 
$\min \sigma_{\rm ess}(A) = 0$, which implies that $\lambda <0$. Let us introduce the $H$-orthogonal spectral projector $P:= \chi_{(-\infty, \lambda/2]}(A)$, where $\chi_{(-\infty, \lambda/2]}$ 
is the characteristic function of the interval $(-\infty, \lambda/2]$. The projector $P$ is finite-rank and its range is equal to the 
subspace of $H$ spanned by the eigenvectors associated to the discrete eigenvalues of $A$ lower or equal to $\lambda/2$. In particular, (\ref{eq:eqw}) implies that $w\in \mbox{Ran}(P)$. 
For all $k\in \N^*$, $u_{n_k}$ can be decomposed as $u_{n_k} = Pu_{n_k} + (1-P)u_{n_k}$. Since $(u_{n_k})_{k\in\N}$ weakly converges in $V$ towards $w$ and $P$ is finite-rank,
$\left(P u_{n_k}\right)_{k\in\N}$ strongly converges in $V$ to $w$ and $\left( (1-P)u_{n_k}\right)_{k\in\N}$ 
weakly converges in $V$ to $0$. In particular, we have $\dps \mathop{\lim}_{k\to +\infty} \|(1-P)u_{n_k}\|^2 = \mathop{\lim}_{k\to +\infty} \|u_{n_k}\|^2 - \|Pu_{n_k}\|^2  = 1 -\|w\|^2$, and 
since for all $k\in \N$, $(1-P)u_{n_k} \in \mbox{\rm Ker} (P)$, 
\begin{equation}\label{eq:lowerbound}
\mathop{\mbox{liminf}}_{k\to + \infty} a\left( (1-P)u_{n_k}, (1-P)u_{n_k} \right) \geq \frac{\lambda}{2} \left( 1 - \|w\|^2 \right). 
\end{equation}
Besides, for all $k\in\N$, it holds that
$$
a\left( u_{n_k}, u_{n_k}\right) = a\left( P u_{n_k}, Pu_{n_k}\right) + a\left( (1-P) u_{n_k}, (1-P) u_{n_k} \right),
$$
with
$$
\mathop{\lim}_{k\to +\infty} a\left( u_{n_k}, u_{n_k}\right) = \lambda \quad \mbox{and} \quad \mathop{\lim}_{k\to +\infty} a\left( Pu_{n_k}, Pu_{n_k}\right) = a(w,w) = \lambda \|w\|^2.
$$
This yields that 
\begin{equation}\label{eq:greaterbound}
\mathop{\lim}_{k\to +\infty} a\left( (1-P)u_{n_k}, (1-P)u_{n_k}\right) = \lambda (1-\|w\|^2).
\end{equation}
Since $0 > \frac{\lambda}{2} > \lambda$, (\ref{eq:lowerbound}) and (\ref{eq:greaterbound}) necessarily imply that $\|w\| = 1$. 

\medskip

Consequently, $\|w\|^2  = 1  = \mathop{\lim}_{k\to + \infty} \|u_{n_k}\|^2$ and $a(w,w) = \lambda = \mathop{\lim}_{k\to + \infty} a(u_{n_k}, u_{n_k})$. Thus, the convergence of the sequence 
$\left( u_{n_k} \right)_{k\in\N}$ towards $w$ also holds strongly in $V$. The rest of the proof then uses exactly the same arguments as those used in the previous section.   
\end{proof}

\subsection{Proof of Theorem~\ref{th:main} and of Proposition~\ref{prop:noHV} for the orthogonal greedy algorithms}\label{sec:prooforth}

It is clear that there always exists at least one solution to the minimization problems (\ref{eq:orthopt}). 

The arguments of the proof are similar for both algorithms. For all $n\in\N^*$, let us denote by $\alpha_n:= \frac{1}{\|u_{n-1} + z_n\|}$, $\zt_n = \alpha_n z_n$, 
$\ut_n:= \alpha_n u_{n-1} + \zt_n$ and $\widetilde{\lambda}_n =a(\ut_n, \ut_n)$.

For all $n\in \N^*$, $\lambda_n = a(u_n, u_n) \leq \widetilde{\lambda}_n = a(\ut_n, \ut_n)$. 
Besides, the same calculations as the ones presented in Section~\ref{sec:proofRayleigh} and Section~\ref{sec:proofResidual} can be carried out, replacing $u_n$ by $\ut_n$. This implies that 
for all $n\in\N^*$, $\widetilde{\lambda}_n \leq \lambda_{n-1}$ (and thus the sequence $(\lambda_n)_{n\in\N}$ is non-increasing). Besides,
 the series of general term $\left( \|\zt_n\|_a^2 \right)_{n\in\N}$ is convergent. 
 
Equations (\ref{eq:eqRanew}) and (\ref{eq:eqRenew}) are still valid for the orthogonalized versions of the algorithms. 
Thus, following exactly the same lines as in Section~\ref{sec:proofRayleigh} and Section~\ref{sec:proofResidual}, we obtain the desired results.
 The fact that for all $n\in \N^*$, $\langle u_n, u_{n-1} \rangle \geq 0$ ensures the uniqueness of the limit of the sequence in the
 case when the eigenvalue $\lambda$ is simple.

Of course, the same kind of arguments as in Section~\ref{sec:proofnoHV} leads to the same conclusion for the OReGA in the case when the embedding $V\hookrightarrow H$ is not assumed to be compact, 
hence Proposition~\ref{prop:noHV}.

\subsection{Proof of Theorem~\ref{prop:finitedim}}\label{sec:prooffinited}

\begin{lemma}
Consider the PRaGA and PReGA in finite dimension. Then, there exists $C \in \R_+$ such that for all $n\in\N$, 
\begin{equation}\label{eq:estgard}
\|\cJ '(u_n)\|_* \leq C \|z_{n+1}\|_a.
\end{equation}
\end{lemma}

Let us recall that the norm $\|\cdot\|_*$ is the injective norm on $V'$ i.e.
$$
\forall l\in V', \; \|l\|_* = \mathop{\sup}_{z\in \Sigma^*} \frac{\langle l, z\rangle_{V',V}}{\|z\|_a},
$$
and that for all $v\in \Omega = \{ u\in V, \; 1/2 < \|u\| < 3/2 \}$, the derivative of $\cJ$ at $v$ is given by
$$
\forall w\in V, \quad \langle \cJ'(v), w \rangle_{V',V}  = \frac{1}{\|v\|^2}\left( a(v,w) - a(v,v) \langle v, w \rangle \right).
$$

\begin{proof}
For the PReGA algorithm, (\ref{eq:estgard}) is straightforward since, using (\ref{eq:algo3}) and Lemma~\ref{lem:Figueroa}, it holds that
$$
\|\cJ '(u_n)\|_* =  \mathop{\sup}_{z\in \Sigma^*} \frac{a(u_n, z) - \lambda_n \langle u_n, z \rangle}{\|z\|_a} =  \mathop{\sup}_{z\in \Sigma^*} \frac{\langle u_n, z\rangle_a - (\nu + \lambda_n) \langle u_n, z \rangle}{\|z\|_a} =  \|z_{n+1}\|_a.
$$
Let us now prove (\ref{eq:estgard}) for the PRaGA. 
Since $\cJ$ is $\cC^2$ on the compact bounded set $\overline{\Omega}$, the Hessian of $\cJ$ at any $v\in \Omega$ is bounded. Thus, since $\|u_n\| = 1$ for all $n\in \N$ and
$\dps z_n \mathop{\longrightarrow}_{n\to +\infty} 0$ strongly in $H$, there exists $C>0$, $n_0\in \N$ and $\varepsilon_0 >0$ such that 
for all $n\geq n_0$, all $\varepsilon\leq \varepsilon_0$ and all $z\in \Sigma$ 
such that $\|z\|_a \leq 1$, 
$$
\cJ(u_{n} + z_{n+1}) \leq \cJ(u_n + \varepsilon z) \leq \cJ(u_n+z_{n+1}) + \langle \cJ '(u_n+z_{n+1}), \varepsilon z - z_{n+1} \rangle_{V',V} + C \|\varepsilon z - z_{n+1}\|_a^2.
$$  
Since $\langle \cJ '(u_n + z_{n+1}), z_{n+1} \rangle_{V',V} = 0$ from Lemma~\ref{lem:ELrayleigh}, 
the above inequality implies that
$$
\varepsilon \left|\langle \cJ '(u_n+z_{n+1}), z \rangle_{V',V}\right| \leq C \| \varepsilon z - z_{n+1}\|_a^2 \leq 2C \left(\varepsilon^2 \|z\|_a^2 + \|z_{n+1}\|_a^2\right). 
$$
Taking $\varepsilon = \frac{\|z_{n+1}\|_a}{\|z\|_a}$ in the above expression yields 
$$
\forall z\in \Sigma, \quad \left|\langle \cJ '(u_n+z_{n+1}), z \rangle_{V',V}\right| \leq 4C \|z\|_a\|z_{n+1}\|_a.
$$
Using again the fact that the Hessian of $\cJ$ is bounded in $\Omega$, and that $\dps \mathop{\lim}_{n\to\infty} \|z_{n+1}\|_a=0$, there exists $n_0\in \N$ such that for all $n\geq n_0$, 
$$
\forall z\in \Sigma, \quad \left|\langle \cJ '(u_n+z_{n+1}), z \rangle_{V',V} -\langle \cJ '(u_n), z \rangle_{V',V}  \right| \leq C \|z\|_a \|z_{n+1}\|_a,
$$
and finally
$$
\forall z\in \Sigma, \quad \left|\langle \cJ '(u_n), z \rangle_{V',V}\right| \leq 5C \|z\|_a\|z_{n+1}\|_a,
$$
which yields the announced result. 
\end{proof}

\begin{proof}[Proof of Theorem~\ref{prop:finitedim}]
Since $\dps d_a(u_n, F_\lambda) \mathop{\longrightarrow}_{n\to\infty} 0$, using (\ref{eq:Lojanew}), there exists $n_0\in \N$ such that for $n\geq n_0$,
$$
\left|\cJ(u_n) - \lambda\right|^{1-\theta} =  \left(\lambda_n - \lambda\right)^{1-\theta} \leq K \|\cJ '(u_n)\|_*.
$$
Thus, using the concavity of the function $\R_+ \ni t \mapsto t^{\theta}$, we have
$$
 (\lambda_n -\lambda)^\theta - (\lambda_{n+1} - \lambda)^{\theta}  \geq  \frac{\theta}{(\lambda_n - \lambda)^{1-\theta}}\left( \lambda_n - \lambda_{n+1}\right) \geq  \frac{\theta}{K \| \cJ'(u_n)\|_* } \left( \lambda_n - \lambda_{n+1}\right).
$$
Equation (\ref{eq:estnew}) implies that $ \lambda_n - \lambda_{n+1} \geq \|\zt_{n+1}\|_a^2$. Besides, since $\|u_n\|^2 = 1$, it holds that for all 
$v\in V$, 
$$
\langle \cJ '(u_n), v \rangle_{V',V} = a(u_n, v) - \lambda_n \langle u_n, v \rangle. 
$$ 

Consequently, for $n$ large enough, using (\ref{eq:estgard}) and the fact that $\dps \alpha_n \mathop{\longrightarrow}_{n\to \infty} 1$, 
\begin{align*}
 (\lambda_n -\lambda)^\theta - (\lambda_{n+1} - \lambda)^{\theta}  &\geq  \frac{\theta}{K \| \cJ'(u_n)\|_* } \left( \lambda_n - \lambda_{n+1}\right) \geq  \frac{\theta}{K C \|z_{n+1}\|_a } \|\zt_{n+1}\|_a^2 \\
 &\geq  \frac{\theta \alpha_{n+1}}{KC } \|\zt_{n+1}\|_a \geq  \frac{\theta}{2KC} \|\zt_{n+1}\|_a.\\
\end{align*}
Since $\dps \mathop{\lim}_{n\to\infty} \alpha_n = 1$ and the series of general term 
$\left( (\lambda_n -\lambda)^\theta - (\lambda_{n+1} - \lambda)^{\theta}\right)_{n\in\N}$ is convergent, 
the series of general terms $(\|\zt_n\|_a)_{n\in\N^*}$ 
and $(\|z_n\|_a)_{n\in\N^*}$ are convergent. 
Besides, since $\alpha_n = \frac{1}{\|u_{n-1} + z_n\|}$, it can be easily seen that $|1-\alpha_n| = \mathcal{O}(\|z_n\|)$ is also the general 
term of a convergent series. Thus, since 
$\|u_n - u_{n-1}\|_a \leq |1-\alpha_n| (\lambda_\Sigma + \nu) + \|\zt_n\|_a$, the sequence $(u_n)_{n\in\N}$ strongly converges in $V$ to some $w\in F_\lambda$. 
This also implies that there exists $c>0$ and $n_0\in \N^*$  such that for all $n\geq n_0$, 
$$
\|u_n - u_{n-1}\|_a \leq c \|\zt_{n}\|_a.
$$
Besides, denoting by $e_n:= \sum_{k=n}^{+\infty} \|\zt_k\|_a$, we have 
\begin{equation}\label{eq:en}
\|u_n -w\|_a \leq \sum_{k=n}^{+\infty} \|u_{k+1} -u_k\|_a \leq c e_n.
\end{equation}

\medskip

Let us now prove the rates (\ref{eq:rate1}) and (\ref{eq:rate2}). The strategy of proof is identical to the one used in \cite{Levitt}. 

\medskip

The above calculations imply that for $k$ large enough, 
\begin{equation}\label{eq:kla}
|\lambda_k -\lambda|^\theta - |\lambda_{k+1} - \lambda|^{\theta} \geq  \frac{\theta}{ACK} \|\zt_{k+1}\|_a,
\end{equation}
for any constant $A >2$. We choose $A$ large enough to ensure that $M = \frac{1}{CK}\left( \frac{\theta}{ACK} \right)^{\frac{1-\theta}{\theta}} <1$. Let us first prove that for all $n\in\N^*$,
\begin{equation}\label{eq:est0}
e_{n+1}  \leq e_n - M e_n^{\frac{1-\theta}{\theta}}
\end{equation}

By summing inequalities (\ref{eq:kla}) for $k$ ranging from $n-1$ to infinity, we obtain
$$
 \frac{\theta}{ACK} e_n \leq |\lambda_{n-1} -\lambda|^\theta,
$$
which yields 
$$
 \left( \frac{\theta}{ACK} e_n \right)^{\frac{1-\theta}{\theta}}  \leq  |\lambda_{n-1} -\lambda|^{1-\theta}
 \leq  K \|\cJ '(u_{n-1})\|_*
  \leq   CK\|\zt_n\|_a
 =  CK (e_n - e_{n+1}).
$$
Hence, (\ref{eq:est0}).

\medskip

If $\theta = \frac{1}{2}$, (\ref{eq:est0}) reduces to
$$
e_{n+1} \leq (1-M)e_n.
$$
Thus, there exists $c_0>0$ such that for all $n\in \N^*$, $e_n \leq c_0 (1-M)^n$. Since we have chosen $A$ large enough so that $0 < 1-M <1$, (\ref{eq:en}) immediately yields (\ref{eq:rate1}). 

\medskip

If $\theta \in (0, 1/2)$, we set $t := \frac{\theta}{1 - 2\theta}$ and, for $n$ large enough, $y_n = Bn^{-t}$ for some constant $B>0$ which will be chosen later. Then, 
$$
 y_{n+1}  =  B (n+1)^{-t} =  Bn^{-t}\left( 1 + \frac{1}{n}\right)^{-t} \geq  B n^{-t} \left( 1 - \frac{t}{n}\right)
 =  y_n \left( 1 -t B^{-1/t}y_n^{1/t}\right).
$$
Then, we choose $B$ large enough so that $B> \left( \frac{M}{t}\right)^{-t}$ with $M = \frac{1}{CK}\left( \frac{\theta}{2CK} \right)^{\frac{1-\theta}{\theta}}$.
 Using (\ref{eq:est0}), we then prove by induction that $e_n \leq y_n$, which yields (\ref{eq:rate2}).
\end{proof}

\section{Appendix: Some pathological cases}\label{sec:appendix}

\begin{example}\label{ex:ex1} {\bf Problem (\ref{eq:minipb}) may have no solution.}

\medskip

{\rm Let $\cH = L^2_{\rm per}(-\pi,\pi)$, and $\cV = H^1_{\rm per}(-\pi,\pi)$, and let $(e_k)_{k\in\Z}$ be the orthonormal basis of $\cH$ defined as:
$$
\forall k\in \Z, \quad e_k:x\in (-\pi,\pi) \mapsto \frac{1}{\sqrt{2\pi}}e^{ikx}.
$$
Let then $H = \cH \otimes \cH$  and $V = \cV \otimes \cV$ so that 
(HV) is satisfied and $(e_k\otimes e_l)_{(k,l) \in \Z^2 }$ forms an orthonormal basis of $H$. It can be easily checked that the set
$$
\Sigma := \left\{ r \otimes s, \; r, s \in \cV \right\}.
$$
satisfies assumptions (H$\Sigma 1$), (H$\Sigma 2$) and (H$\Sigma 3$). Let $(\mu_{k,l})_{(k,l)\in\Z^2}$ be a set of real numbers such that for all $(k,l)\neq (k',l')\in \Z^2$, 
we have $\mu_{k,l} \neq \mu_{k',l'}$. 
Let $a:H\times H \to \R$ be the unique symmetric bilinear form such that:
$$
\forall v\in V, \quad \forall (k,l)\in \Z^2, \quad a(\psi_{k,l}, v) = \mu_{k,l} \langle \psi_{k,l}, v \rangle,
$$
where
$$
\psi_{0,1} :=  \frac{e_0\otimes e_1 + e_1 \otimes e_0}{\sqrt{2}}, \quad
\psi_{1,0} :=  \frac{e_0\otimes e_1 - e_1 \otimes e_0}{\sqrt{2}}, \quad
\forall (k,l)\in \Z^2\setminus\{(1,0), (0,1)\}, \; \psi_{k,l} := e_k\otimes e_l.
$$
We choose the sequence $(\mu_{k,l})_{(k,l)\in \Z^2}$ such that
$$
0 < \mu_{0,1} < \mu_{0,0} < \mu_{1,0}  < M,
$$
for some constant $M>0$ and for all $(k,l)\in \Z^2\setminus\{(0,1),(0,0),(1,0)\}$, 
$$
M + 0.5(1 + |k|^2)(1+|l|^2)\leq \mu_{k,l} \leq M + (1 + |k|^2)(1+|l|^2).
$$
Thus, the lowest eigenvalue of $a(\cdot, \cdot)$ is $\mu_{0,1}$ and an associated eigenvector is $\psi_{0,1}$. 
The bilinear form $a$ is continuous on $V = \cV\times \cV$ and satisfies (HA). Besides, since for all 
$(k,l)\neq (k',l') \in \Z^2$, we have $\mu_{k,l} \neq \mu_{k',l'}$, it holds
$$
a(\psi_{k,l}, \psi_{k',l'}) = \mu_{k,l} \langle \psi_{k,l}, \psi_{k',l'} \rangle = \mu_{k',l'} \langle \psi_{k,l}, \psi_{k',l'} \rangle = 0.
$$
Let $w=e_0 \otimes e_0$. For all $m\in\N^*$, let $z_m:= - \left(e_0 + \frac{1}{m}e_1\right)\otimes \left(e_0 + \frac{1}{m}e_1\right)$. For all $m\in \N^*$, $z_m \in \Sigma$, 
$w+z_m = - \frac{1}{m}\left( e_1 \otimes e_0 + e_0 \otimes e_1 + \frac{1}{m}e_1 \otimes e_1 \right)$.
$$
\cJ(w + z_m) \mathop{\longrightarrow}_{m\to \infty} \cJ (e_1\otimes e_0 + e_0 \otimes e_1)= \mu_{0,1}.
$$
The sequence $(z_m)_{m\in \N^*}$ is then a minimizing sequence of problem (\ref{eq:minipb}) since $\mu_{0,1} = \mathop{\inf}_{v\in V} \cJ(v)$.  

Thus, if there were a minimizer $z_0 = r_0 \otimes s_0\in \Sigma$, with $r_0, s_0 \in \cV$, to problem (\ref{eq:minipb}), 
necessarily $\alpha w + r_0\otimes s_0 = \pm (e_1 \otimes e_0 + e_0 \otimes e_1)$ for some normalization constant $\alpha >0$,
 which is not possible since $w = e_0\otimes e_0$.}
\end{example}

\medskip

\begin{example}\label{ex:ex2} {\bf The greedy algorithms may converge to ``excited'' states.}

\medskip

{\rm Let us take the same notation as in Example~\ref{ex:ex1} and define this time the bilinear form $a(\cdot, \cdot)$ as the unique symmetric bilinear form such that
$$
\forall v\in V, \quad \forall (k,l)\in \Z^2, \quad a(\psi_{k,l}, v) = \mu_{k,l} \langle \psi_{k,l}, v\rangle,
$$
where
$$
 \psi_{0,2}:= \frac{e_0\otimes e_2 + e_2 \otimes e_0}{\sqrt{2}}, \quad
\psi_{2,0}: = \frac{e_0\otimes e_2 - e_2 \otimes e_0}{\sqrt{2}}, \quad
\forall (k,l)\in\Z^2\setminus\{(0,1), (1,0)\}, \; \psi_{k,l}:= e_k\otimes e_l.
$$
We choose the sequence $(\mu_{k,l})_{(k,l)\in \Z^2}$ such that
$$
0 < \mu_{0,2} < \mu_{1,1} < \mu_{2,0} < M,
$$
for some constant $M>0$ which will be chosen later, and for all $(k,l)\in \Z^2 \setminus\{(0,2),(2,0),(1,1)\}$,
$$
M + 0.5 (1 + |k|^2)(1+|l|^2) \leq \mu_{k,l} \leq M + (1 + |k|^2)(1+|l|^2),
$$
so that for all $(k,l)\neq(k',l')\in \Z^2$, $\mu_{k,l} \neq \mu_{k',l'}$. The smallest eigenvalue of $a(\cdot, \cdot)$ is then $\mu_{0,2}$ and an associated eigenvector is 
$\psi_{0,2}$. It is easy to check that $a(\cdot, \cdot)$ still satisfies (HA). 
Let us prove that
$$
e_1 \otimes e_1 \in \mathop{\mbox{\rm argmin}}_{z\in \Sigma} \cJ(z),
$$
so that $\lambda_\Sigma = \mu_{1,1} > \mu_{0,2} = \mathop{\inf}_{v\in V} \cJ(v)$. 
Let us argue by contradiction and assume that there exists $r,s \in \cV$ such that $\|r\otimes s\| = 1$ and $\cJ(r\otimes s) < \cJ(e_1\otimes e_1)$. Since $(e_k)_{k\in\Z}$ forms an orthonormal basis 
of $\cH$, we can choose $r,s\in \cV$ such that there exists two sequences of real numbers $(c_k^r)_{k\in\Z}$ and $(c_k^s)_{k\in\Z}$ such that
$$
r = \sum_{k\in \Z} c_k^r e_k, \quad s = \sum_{k\in\Z} c_k^s e_k, \quad \sum_{k\in\Z}\left|c_k^r\right|^2 =  \sum_{k\in\Z}\left|c_k^s\right|^2  = 1.
$$
It is easy to check that if $\cJ(r\otimes s) < \cJ(e_1\otimes e_1)$, then, necessarily, 
\begin{equation}\label{eq:contrad}
 \cJ\left( \left( c_0^r e_0 + c_2^r e_2\right)\otimes \left(c_0^s e_0 + c_2^s e_2 \right)\right) < \cJ(e_1\otimes e_1) = \mu_{1,1} > \mu_{0,2}.
\end{equation}
Let us now prove that, up to carefully choosing the values of the eigenvalues $(\mu_{k,l})_{(k,l)\in\Z^2}$, it may happen that 
$\cJ\left( \left( \cos \theta e_0 + \sin \theta e_2\right) \otimes \left( \cos \phi e_0 + \sin \phi e_2\right) \right) \geq \mu_{1,1}$ 
for any $\theta, \phi \in \R$, which will yield a 
contradiction. 
It holds
\begin{align*}
 z_{\theta,\psi}& := \left( \cos \theta e_0 + \sin \theta e_2\right) \otimes \left( \cos \phi e_0 + \sin \phi e_2\right) \\
& = \cos\theta \cos \phi \psi_{0,0} + \sin\theta \sin \phi \psi_{2,2} + \frac{\sqrt{2}}{2} (\cos \theta \sin \phi + \cos \phi \sin \theta) \psi_{0,2} \\
& + \frac{\sqrt{2}}{2} (\cos \theta \sin \phi - \cos \phi \sin \theta) \psi_{2,0},\\
\end{align*}
so that
\begin{align*}
a(z_{\theta,\psi}, z_{\theta, \psi})  =  \cJ(z_{\theta,\phi}) 
&= \cos^2\theta \cos^2 \phi \mu_{0,0} + \sin^2\theta \sin^2 \phi \mu_{2,2} + \frac{1}{2}(\cos \theta \sin \phi + \cos \phi \sin \theta)^2 \mu_{0,2} \\
&+ \frac{1}{2}(\cos \theta \sin \phi - \cos \phi \sin \theta)^2 \mu_{2,0}.\\
\end{align*}
We want to prove that
\begin{equation}\label{eq:interm}
\forall (\theta,\phi)\in \R^2, \quad \cJ(z_{\theta, \phi}) \geq \mu_{1,1} = \mu_{1,1}(\cos^2\theta + \sin^2\theta)(\cos^2\phi + \sin^2 \phi).
\end{equation}
Since $(\theta, \phi) \in \R^2 \mapsto \cJ(z_{\theta, \phi})$ is continuous, it is sufficient to consider
 $\theta, \phi \in \R$ such that $\cos \theta\cos \phi \neq 0$. Denoting by 
$t_1:=\tan \theta$ and $t_2:= \tan \phi$, proving (\ref{eq:interm}) amounts to proving that
$$
\forall t_1,t_2\in\R, \quad A(t_1, t_2):= \mu_{0,0} + t_1^2 t_2^2 \mu_{2,2} + \frac{1}{2}\mu_{0,2} (t_1 + t_2)^2 + \frac{1}{2}\mu_{2,0} (t_1- t_2)^2 - \mu_{1,1} (1+ t_1^2)(1+t_2^2) \geq 0.
$$
The quantity $A(t_1,t_2)$ can be rewritten as
\begin{align*}
 A(t_1,t_2) & =  (\mu_{0,0} - \mu_{1,1}) + (\mu_{2,2} - \mu_{1,1})t_1^2 t_2^2 + \frac{1}{2}\mu_{0,2} (t_1^2 + t_2^2) + \frac{1}{2}(\mu_{2,0} - \mu_{0,2})(t_2 -t_1)^2 - \mu_{1,1} (t_1^2 + t_2^2)\\
& = \frac{1}{2}\mu_{0,2} (t_1^2 + t_2^2) +  \frac{1}{2}(\mu_{2,0} - \mu_{0,2} - 2\mu_{1,1})(t_2 -t_1)^2 - \mu_{1,1} (t_1^2 + t_2^2) -2\mu_{2,0} t_1t_2 \\
&+ (\mu_{0,0} - \mu_{1,1}) + (\mu_{2,2} - \mu_{1,1})t_1^2 t_2^2 .\\
\end{align*}
Thus by choosing $\mu_{2,0} > \mu_{0,2} + 2\mu_{1,1}$, and $\mu_{0,0}$, $\mu_{2,2}$ and $M$ large enough, $A(t_1,t_2)$ is 
ensured to be non-negative for any $t_1,t_2\in\R$. Indeed, by setting $x = t_1t_2$, 
the last three terms in the last equality may be rewritten as a second degree polynomial in $x$ whose discriminant is negative.
This leads to a contradiction with (\ref{eq:contrad}).

\medskip

Thus, $\dps \psi_{1,1} = e_1\otimes e_1 \in \mathop{\mbox{argmin}}_{z\in \Sigma} \cJ(z)$ and since $\psi_{1,1}$ is an eigenvector of $a(\cdot, \cdot)$, 
the sequence $(\lambda_n)_{n\in\N}$ 
generated by the PRaGA or the PReGA is such that for all $n\in\N$, $\lambda_n = \lambda_\Sigma= \mu_{1,1} > \mu_{0,2} = \mathop{\inf}_{v\in V} \cJ(v)$. 
The sequence $(\lambda_n)_{n\in\N}$ thus does not converge towards the lowest eigenvalue of $a(\cdot, \cdot)$. }   
\end{example}

We prove here the result announced in Section~\ref{sec:initRay}. 

\begin{lemma}\label{lem:onebody}
 Let $\cX_1, \; \ldots, \;\cX_d$ be bounded regular domains of $\R^{m_1},\; \ldots,\; \R^{m_d}$ respectively. Let $V = H^1_0(\cX_1\times \cdots \times \cX_d)$,
 $H = L^2(\cX_1\times \cdots \times \cX_d)$, $W\in \cC^{0,\alpha}\left(\cX_1\times \cdots \times \cX_d, \R\right)$ for some $0 < \alpha <1$ and 
$a:V\times V \to \R$ be the continuous bilinear form defined by
$$
\forall v,w\in V, \quad a(v,w):= \int_{\cX_1\times \cdots \times \cX_d} \nabla v \cdot \nabla w + W vw.
$$
For all $1\leq j \leq d$, let $V_j:= H^1_0(\cX_j)$ and let $\Sigma^{\otimes}$ be the set of rank-1 tensor product functions defined by (\ref{eq:rank1}).
 Then, the two following assertions are equivalent:
\begin{itemize}
 \item [(i)] there exists an eigenvector $z$ of the bilinear form $a(\cdot, \cdot)$ which belongs to the set of rank-1 tensor product functions $\Sigma^\otimes$; 
\item[(ii)] the potential $W$ is the sum of $d$ one-body potentials, i.e. there exist $W_1 \in \cC^{0, \alpha}(\cX_1, \R),\; \ldots,\; W_d \in \cC^{0, \alpha}(\cX_d, \R)$ such that
$$
W(x_1, \ldots, x_d) = W_1(x_1) + \cdots + W_d(x_d).
$$
\end{itemize}
\end{lemma}

\begin{proof}
The fact that $(ii)\Rightarrow (i)$ is obvious. Indeed, if for all $1\leq j \leq d$, 
$r_j$ is an eigenvector of the continuous symmetric bilinear form $a_j: V_j \times V_j \to \R$ defined by 
$$
\forall v_j, w_j \in V_j, \quad a_j\left( v_j, w_j\right):= \int_{\cX_j} \nabla v_j \cdot \nabla w_j + \int_{\cX_j}W_j  v_jw_j,
$$ 
with respect to the scalar product of $L^2(\cX_j)$, then the rank-1 tensor product function $z = r_1\otimes \cdots \otimes r_d \in \Sigma$ is an eigenvector 
of the bilinear form $a(\cdot, \cdot)$ 
with respect to the $L^2(\cX_1 \times \cdots \times \cX_d)$. 

\medskip

Let us now prove that $(i)\Rightarrow (ii)$. Let $z\in \Sigma$ be an $H$-normalized eigenvector of $a(\cdot, \cdot)$ associated with an eigenvalue $\lambda$. 
There exist $r_1 \in V_1,\; \ldots ,\; r_d\in V_d$ such that 
$z = r_1 \otimes \cdots \otimes r_d$. Without loss of generality, we can assume that $\left\|r_1\right\|_{L^2(\cX_1)} = \cdots = \left\|r_d\right\|_{L^2(\cX_d)} = 1$. Since the potential 
$W$ is assumed to be $\cC^{0, \alpha}$, then for all $1\leq j \leq d$, $r_j\in \cC^{2, \alpha}(\cX_j, \R)$ and it holds that for all $(x_1, \ldots, x_d)\in \cX_1 \times \cdots \times \cX_d$, 
\begin{align}\label{eq:initeigpb}
&-\Delta r_1(x_1)r_2(x_2) \cdots r_d(x_d) - r_1(x_1) \Delta r_2(x_2) \cdots r_d(x_d) - r_1(x_1) r_2(x_2) \cdots \Delta r_d(x_d) \nonumber\\ 
& + W(x_1, \ldots, x_d) r_1(x_1) \cdots r_d(x_d) = \lambda r_1(x_1) \cdots r_d(x_d). \nonumber\\
\end{align}
For all $1\leq j \leq d$, multiplying the above equation by $r_1(x_1)\cdots r_{j-1}(x_{j-1})r_{j+1}(x_{j+1})\cdots r_{d}(x_d)$ and integrating over 
$\cX_1 \times \cdots \times \cX_{j-1} \times \cX_j \times \cdots \times \cX_d$ leads to
\begin{equation}\label{eq:Wj}
-\Delta r_j(x_j) + W_j(x_j)r_j(x_j) = \lambda r_j(x_j),
\end{equation}
where for all $1\leq j \leq d$ and $x_j\in \cX_j$,
\begin{align*}
W_j(x_j)&:= \int_{\cX_1\times \cdots \times \cX_{j-1} \times \cX_{j+1} \times \cdots \times \cX_d} W(x_1, \ldots, x_d)\left|r_1(x_1) \cdots r_{j-1}(x_{j-1})r_{j+1}(x_{j+1}) \cdots r_{d}(x_d)\right|^2\,dx_1\cdots\,dx_{j-1}\,dx_{j+1}\cdots\,dx_d\\
& + \int_{\cX_1} \left|\nabla r_1(x_1)\right|^2\,dx_1 + \cdots + \int_{\cX_{j-1}} \left|\nabla r_{j-1}(x_{j-1})\right|^2\,dx_{j-1} + \int_{\cX_{j+1}} \left|\nabla r_{j+1}(x_{j+1})\right|^2\,dx_{j+1} + \cdots +  \int_{\cX_d} \left|\nabla r_{d}(x_d)\right|^2\,dx_d.\\ 
 \end{align*}
The regularity of the functions $W,\; r_{1}, \; \ldots,\; r_{d}$ implies that $W_j$ is $\cC^{0,\alpha}$. Multiplying (\ref{eq:Wj}) by 
$$
r_{1}(x_1)\cdots r_{j-1}(x_{j-1})r_{j+1}(x_{j+1})\cdots r_{d}(x_d)
$$ 
and summing all the obtained equations for $1 \leq j \leq d$, we obtain
\begin{align*}
&-\Delta r_{1}(x_1)r_{2}(x_2) \cdots r_{d}(x_d) - r_{1}(x_1) \Delta r_{2}(x_2) \cdots r_{d}(x_d) - r_{1}(x_1) r_{2}(x_2) \cdots \Delta r_d(x_d) \nonumber\\ 
& + \left(W_1(x_1) + \cdots + W_d(x_d)\right)r_{1}(x_1) \cdots r_{d}(x_d) = d\lambda r_{1}(x_1) \cdots r_{d}(x_d). \nonumber\\
\end{align*}
Subtracting this equation to (\ref{eq:initeigpb}) we obtain that for all $(x_1, \ldots, x_d)\in \cX_1 \times \cdots \times \cX_d$, 
$$
\left(W(x_1, \ldots x_d) - W_1(x_1) - \cdots - W_d(x_d) + (d-1)\lambda \right)r_{1}(x_1) \cdots r_{d}(x_d) = 0.
$$
Since $z = r_{1}\otimes \cdots \otimes r_{d}$ is an eigenfunction of the operator $-\Delta + W$ with homogeneous Dirichlet boundary conditions on $\cX_1 \times \cdots \times \cX_d$, and since 
$W$ is $\cC^{0, \alpha}$, from the Courant Nodal Theorem~\cite{Nodal}, it holds that the Lebesgue measure of the nodal set of $z$ is zero. Thus, 
the regularity of the functions $W, \; W_1, \; \ldots, \; W_d$ implies 
that $W(x_1, \ldots, x_d) = W_1(x_1) + \cdots + W_d(x_d) - (d-1) \lambda$. Hence $(ii)$.   
\end{proof}

\section{Acknowledgments}

This work has been done while E.C. and V.E. were long-term visitors at IPAM (UCLA). The authors would like to thank 
Sergue\"i Dolgov, Venera Khoromska\"ia and Boris Khoromskij for interesting discussions. 

\bibliography{biblio}

\end{document}